\documentclass[12pt]{article}

\usepackage[T1]{fontenc}
\usepackage{lmodern}
\usepackage{amsmath,amssymb,amsfonts,amsthm,mathtools}
\usepackage{enumitem}
\usepackage[pdfauthor={Chunqiu Fang and Rongxing Xu},
  pdftitle={Codegree Thresholds for lambda-Choosability of Graphs},
  pdfstartview=XYZ,bookmarks=true,colorlinks=true,
  linkcolor=blue,urlcolor=blue,citecolor=blue,
  linktocpage=true,hyperindex=true]{hyperref}

\allowdisplaybreaks
\numberwithin{equation}{section}
\setlist[enumerate]{leftmargin=2.35em,itemsep=0.2em,topsep=0.3em}

\newtheorem{theorem}{Theorem}[section]
\newtheorem{lemma}[theorem]{Lemma}
\newtheorem{proposition}[theorem]{Proposition}
\newtheorem{corollary}[theorem]{Corollary}
\newtheorem{definition}[theorem]{Definition}

\theoremstyle{definition}

\newcommand{\Prob}{\mathbb{P}}
\newcommand{\Ex}{\mathbb{E}}
\newcommand{\Use}{\operatorname{Use}}
\newcommand{\Bin}{\operatorname{Bin}}

\newcommand{\ee}{\mathrm{e}}
\newcommand{\equalparts}[2]{\{#1\star #2\}}

\title{\large{\bf Codegree Thresholds for $\lambda$-Choosability of Graphs}}
\author{
  Chunqiu Fang \thanks{School of Computer Science and Technology, Dongguan University of Technology, Dongguan, Guangdong, 523808, China. Email: \texttt{fcq15@tsinghua.org.cn}}
  \quad
  Rongxing Xu\thanks{School of Mathematical Sciences, Zhejiang Normal University, Jinhua, Zhejiang, 321000, China. Email: \texttt{xurongxing@zjnu.edu.cn}}
}

\date{\today}

\begin{document}
\maketitle
\begin{abstract}
Let $\lambda = \{k_1,\ldots,k_q\}$ be a partition, and let
$|\lambda| = k_1 + \cdots + k_q$ be its size. A $|\lambda|$-list assignment $L$ of
a graph $G$ is a $\lambda$-assignment if the color set
$\bigcup_{v \in V(G)}L(v)$ can be partitioned into $q$ pairwise disjoint sets
$X_1,\ldots,X_q$ such that
$|L(v) \cap X_i| = k_i$ for every vertex $v$ and every
$i = 1,\ldots,q$. This notion,
introduced by Zhu [J. Combin. Theory Ser. B, 2020], puts ordinary coloring and
list coloring in the same framework. The case
$\lambda = \{1,\ldots,1\}$ is ordinary $q$-colorability, while the
case $\lambda = \{k\}$ is ordinary $k$-choosability. A classical theorem of
Alon [Random
Structures Algorithms, 2000] states that every graph with minimum degree $d$
has choice number at least $(1/2 - o(1))\log_2 d$. The constant $1/2$ was later
improved by Saxton and Thomason [Invent. Math., 2015], using the hypergraph
container method, to the sharp constant $1$. It is natural to ask whether a similar
phenomenon holds for $\lambda$-choosability for an arbitrary fixed partition
$\lambda$. If one only assumes large minimum degree, then the answer is negative.
Indeed, balanced complete bipartite graphs have arbitrarily large minimum
degree but are always $\{1,1\}$-choosable. Nevertheless, we show that the minimum
$q$-codegree, defined for a graph $G$ with $|V(G)|\geq q$ by
$\delta_q(G) = \min\{|N_G(S)| : S \subseteq V(G),\ |S| = q\}$, is the
appropriate replacement for minimum degree in this setting.

We prove that for every partition
$\lambda = \{k_1,\ldots,k_q\}$ there exists an integer $d$ such that every graph
$G$ with $\delta_q(G) \geq d$ is not $\lambda$-choosable. Let $f(\lambda)$
be the least such $d$. For every fixed $q$, we also obtain
$f(\lambda)\leq 2^{(2q+o(1))|\lambda|}$ as $|\lambda|\to\infty$. For every
$\lambda$, we prove that
$f(\lambda)\geq \frac{1}{q+1}(1+1/q)^{|\lambda|}$. For equal-part partitions, we
then determine the threshold asymptotically. Let
$\equalparts{k}{q}$ denote the partition with $q$ parts all equal to $k$, and let
$\rho_q$ be the unique $x\in(0,1)$ satisfying
$x=(1-x)^q$. We prove that for every fixed $q$,
$f(\equalparts{k}{q})=\rho_q^{-(1+o(1))k}$ as $k\to\infty$.
For a graph $G$, let its $q$-partition choice number
$\operatorname{ch}^{(q)}(G)$ be the least $k$ such that $G$ is
$\equalparts{k}{q}$-choosable. We further prove that for every fixed
$q$, 
\[
\operatorname{ch}^{(q)}(G)\geq
\left(\frac{1}{\ln(1/\rho_q)}-o(1)\right)\ln\delta_q(G)
\] 
as $\delta_q(G)\to\infty$, and the
coefficient $1/\ln(1/\rho_q)$ is asymptotically best possible. When $q=1$,
this becomes $\operatorname{ch}(G)\geq(1-o(1))\log_2\delta(G)$.
\end{abstract}


\section{Introduction}

All graphs in this paper are finite and simple. We refer to \cite{BondyMurty2008} for graph-theoretic terminology and notation not defined here. For a positive integer $k$, let $[k]=\{1,2,\ldots,k\}$.

List coloring was introduced independently by Vizing~\cite{Vizing1976} and
Erd\H{o}s, Rubin and Taylor~\cite{ERT1979}. A list assignment $L$ of a graph
$G$ assigns a set $L(v)$ of colors to each vertex $v$. An $L$-coloring is a
proper coloring $\phi$ such that $\phi(v) \in L(v)$ for every vertex $v$. For a
positive integer $k$, the graph $G$ is $k$-choosable if it has an $L$-coloring
for every list assignment $L$ with $|L(v)| = k$ for every vertex $v$. The choice
number $\operatorname{ch}(G)$ is the minimum such $k$.

The chromatic number and the choice number can be very different. The complete
bipartite graph $K_{n,n}$ is 2-colorable, while
$\operatorname{ch}(K_{n,n}) = (1 + o(1))\log_2 n$ as $n \to \infty$
\cite{ERT1979}. Alon~\cite{Alon1992} proved more generally that the choice
number of the complete $r$-partite graph with $n$ vertices in each part has
order $r\log_2 n$. He later proved that every graph with minimum degree $d$ has choice
number at least $(1/2 - o(1))\log_2 d$~\cite{Alon2000}. Saxton and
Thomason~\cite{ST2015} introduced the hypergraph container method and used it
to replace the constant $1/2$ by the best possible constant $1$. The same
method was introduced independently by Balogh, Morris and
Samotij~\cite{BMS2015}. Thus the choice number tends to infinity with the
minimum degree, although the chromatic number need not do so.

Several weaker forms of list coloring have also been studied.
Kratochv{\'\i}l, Tuza and Voigt~\cite{KTV1998Complexity} introduced separation
choosability, in which the lists of adjacent vertices are required to have
only a small intersection. Kr\'al' and Sgall~\cite{KralSgall2005} introduced
list coloring with a bounded palette, in which all lists use a common palette
of bounded size. For Alon-type results on these variants, especially for
bipartite graphs and graphs of large minimum degree, see
\cite{KTV1998Brooks,FKK2014,EKT2019,Kang2013,BonamyKang2017}.

In this paper, we consider another variant of list coloring, called
$\lambda$-choosability, which was introduced recently by
Zhu~\cite{Zhu2020}. Let
$\lambda = \{k_1,\ldots,k_q\}$ be a partition, where the $k_i$ need not be distinct, and let $|\lambda| = k_1 + \cdots + k_q$. A
$|\lambda|$-list assignment $L$ is a $\lambda$-assignment if its color set can
be partitioned into $q$ pairwise disjoint sets $X_1,\ldots,X_q$ such that
$|L(v) \cap X_i| = k_i$ for every vertex $v$ and every $i \in [q]$. The graph $G$ is
$\lambda$-choosable if it is $L$-colorable for every $\lambda$-assignment
$L$. This notion puts ordinary coloring and list coloring in the same
framework. For positive integers $a$ and $b$, let $\{a \star b\}$ denote the
partition with $b$ parts, all equal to $a$. The cases $\{k\}$ and
$\{1 \star k\}$ are ordinary $k$-choosability and ordinary
$k$-colorability, respectively. This notion has been studied extensively in recent years. See
\cite{GuZhu2023,ZhuZhu2021,GuJiangWoodZhu2023,ZhuZhu2025} for recent work.

It is natural to ask whether an analogue of Alon's result holds for
$\lambda$-choosability. For a fixed partition $\lambda$, does every graph of
sufficiently large minimum degree fail to be $\lambda$-choosable? The answer
is negative when $\lambda$ has more than one part. For example, $K_{n,n}$ has
minimum degree $n$ and is $\{1,1\}$-choosable. More generally, if $\lambda$
has $q \geq 2$ parts, then the complete $q$-partite graph with $n$ vertices in
each part has minimum degree $(q - 1)n$ and is $q$-colorable, and hence
$\lambda$-choosable. Thus a graph may be
$\lambda$-choosable even when its minimum degree is arbitrarily large. We
prove that a graph is not $\lambda$-choosable if every set of
$q$ vertices has sufficiently many common neighbors.

Formally, for a graph $G$ and a set $S \subseteq V(G)$, let
$N_G(S) = \bigcap_{v \in S} N_G(v)$, where
$N_G(\varnothing)=V(G)$. For an integer $t$ with
$1\leq t\leq |V(G)|$, the
\emph{minimum $t$-codegree} of $G$ is defined by
\[
 \delta_t(G)
 = \min\bigl\{|N_G(S)| : S \subseteq V(G),\ |S|=t\bigr\}.
\]
When the graph is clear, we use $N(S)$ for $N_G(S)$. If $|S| \leq t$, then
$|N_G(S)| \geq \delta_t(G)$.

For a partition $\lambda=\{k_1,\ldots,k_q\}$, let $f(\lambda)$ be the least
integer $d$ such that every graph $G$ with $|V(G)|\geq q$ and
$\delta_q(G)\geq d$ is not $\lambda$-choosable, and let
$f(\lambda)=\infty$ if no such integer exists.
Our first main result is the following.

\begin{theorem}\label{thm:existence}
Let $\lambda=\{k_1,\ldots,k_q\}$ be a partition. There exists an integer
$D_\lambda$ such that every graph $G$ with $\delta_q(G)\geq D_\lambda$
is not $\lambda$-choosable, that is, $f(\lambda)<\infty$. Moreover,
$f(\lambda)\leq 2^{(2q+o(1))|\lambda|}$  as $|\lambda|\to\infty$
whenever $q$ is fixed.
\end{theorem}

For positive integers $r$ and $n$, let $T_r(n)$ denote the complete
$r$-partite graph with $n$ vertices in each part. The $q$-codegree in
Theorem~\ref{thm:existence} cannot be replaced by $p$-codegree for any $p < q$.
Indeed, the graph $T_q(n)$ is
$q$-colorable and hence $\lambda$-choosable for every partition $\lambda$ with
$q$ parts, while $\delta_p(T_q(n)) = (q - p)n$ for $p < q$.
For each positive integer $q$, let $\rho_q$ denote the unique $x\in(0,1)$
satisfying $x=(1-x)^q$. This number is well defined since
$(1-x)^q-x$ is strictly decreasing from $1$ to $-1$ on $[0,1]$.
We obtain the following two lower bounds from complete $(q+1)$-partite graphs.

\begin{proposition}\label{prop:lower-threshold}
Let $\lambda=\{k_1,\ldots,k_q\}$ be a partition. If $N_1$ is a positive integer satisfying
$N_1(q+1)(q/(q+1))^{|\lambda|}<1$,
then $T_{q+1}(N_1)$ is $\lambda$-choosable. Consequently,
\[
f(\lambda)\geq\left\lceil(q+1)^{-1}(1+\frac{1}{q})^{|\lambda|}\right\rceil.
\]
Moreover, if $k$, $q$ and $N_2$ are positive integers satisfying
$(q+1)N_2\rho_q^k<1$,
then $T_{q+1}(N_2)$ is $\equalparts{k}{q}$-choosable. Consequently,
\[
f(\equalparts{k}{q})\geq\left\lceil \frac{\rho_q^{-k}}{(q+1)}\right\rceil.
\]
\end{proposition}

By Theorem~\ref{thm:existence} and the first part of
Proposition~\ref{prop:lower-threshold}, for every fixed $q$ we have
\[
 2^{(\log_2(1+1/q)-o(1))|\lambda|}
 \leq f(\lambda)
 \leq 2^{(2q+o(1))|\lambda|}
\]
as
$|\lambda|\to\infty$. By the equal-part case of
Proposition~\ref{prop:lower-threshold}, we obtain a stronger lower bound. We
then establish a matching upper bound.

\begin{theorem}\label{thm:equal-upper}
Let $q\geq 1$ be fixed. We have
$f(\equalparts{k}{q}) \leq\rho_q^{-(1+o(1))k}$ as $k\to\infty$.
\end{theorem}

By Theorem~\ref{thm:equal-upper} and the second part of
Proposition~\ref{prop:lower-threshold}, we obtain the asymptotic threshold.

\begin{corollary}\label{cor:equal-threshold}
For every fixed $q\geq 1$, $f(\equalparts{k}{q})=\rho_q^{-(1+o(1))k}$ as $k\to\infty$. 
\end{corollary}

We also use the equal-part result to study a corresponding choice parameter.
For a graph $G$, let its \emph{$q$-partition choice number} be
\[
 \operatorname{ch}^{(q)}(G)
 =\min\{k:G\text{ is }\equalparts{k}{q}\text{-choosable}\}.
\]
Thus $\operatorname{ch}^{(1)}(G)=\operatorname{ch}(G)$, and
$\operatorname{ch}^{(q)}(G)=1$ if and only if $G$ is $q$-colorable.

\begin{corollary}\label{cor:partition-choice}
For every fixed $q\geq1$,
\[
 \operatorname{ch}^{(q)}(G)\geq
 \left(\frac{1}{\ln(1/\rho_q)}-o(1)\right)\ln\delta_q(G)
\]
as $\delta_q(G)\to\infty$. The coefficient
$1/\ln(1/\rho_q)$ is asymptotically best possible.
\end{corollary}

\begin{proof}
Let $d=\delta_q(G)$ and fix a sufficiently small $\varepsilon>0$. Take
$k=\lfloor(1-\varepsilon)\ln d/\ln(1/\rho_q)\rfloor$.
Then $k\to\infty$ as $d\to\infty$. By Corollary~\ref{cor:equal-threshold}, $f(\equalparts{k}{q})\leq \rho_q^{-(1+\varepsilon/2)k}$ for sufficiently large $d$.
Since $k\ln(1/\rho_q)\leq(1-\varepsilon)\ln d$, we have
\[
 \begin{aligned}
 \ln f(\equalparts{k}{q}) \leq \left(1+\frac{\varepsilon}{2}\right)
 k\ln(1/\rho_q) \leq \left(1+\frac{\varepsilon}{2}\right)
 (1-\varepsilon)\ln d
 <\ln d.
 \end{aligned}
\] 
Thus $f(\equalparts{k}{q})<d$, so $G$ is not
$\equalparts{k}{q}$-choosable. It follows that
\[
 \operatorname{ch}^{(q)}(G)\geq k+1
 >(1-\varepsilon)\frac{\ln d}{\ln(1/\rho_q)}.
\]
The lower bound follows since $\varepsilon$ can be chosen arbitrarily small.

To see that the coefficient is best possible, let
$N=\lceil\rho_q^{-k}/(q+1)\rceil-1$ and let $G=T_{q+1}(N)$.
By the definition of $N$, we have $(q+1)N\rho_q^k<1$ whenever $N$ is positive,
and $N$ is positive for all sufficiently large $k$.
By Proposition~\ref{prop:lower-threshold}, the graph $G$ is
$\equalparts{k}{q}$-choosable. Hence
$\operatorname{ch}^{(q)}(G)\leq k$, while $\delta_q(G)=N$. Since
$\ln N=(1+o(1))k\ln(1/\rho_q)$,
\[
 \operatorname{ch}^{(q)}(G)
 \leq\left(\frac{1}{\ln(1/\rho_q)}+o(1)\right)
 \ln\delta_q(G).
\]
The coefficient is therefore asymptotically best possible.
\end{proof}

When $q=1$, we have $\rho_1=1/2$,
$\operatorname{ch}^{(1)}(G)=\operatorname{ch}(G)$, and
$\delta_1(G)=\delta(G)$. Corollary~\ref{cor:partition-choice} therefore
becomes $\operatorname{ch}(G)\geq(1-o(1))\log_2\delta(G)$. By the standard
minimum-degree subgraph argument, every graph $G$ of average degree $d$
contains a subgraph $H$ with $\delta(H)\geq d/2$. Hence
$\operatorname{ch}(G)\geq\operatorname{ch}(H)
\geq(1-o(1))\log_2(d/2)=(1-o(1))\log_2d$. This is the asymptotically sharp
lower bound in terms of average degree proved by Saxton and
Thomason~\cite{ST2015}.

The rest of the paper is organized as follows. In
Section~\ref{sec:complete-multipartite}, we prove
Proposition~\ref{prop:lower-threshold}. In Section~\ref{sec:surrounded}, we
prove Theorem~\ref{thm:existence}.
In Section~\ref{sec:partition}, we prove Theorem~\ref{thm:equal-upper}.
 
\section{Proof of Proposition~\ref{prop:lower-threshold}}
\label{sec:complete-multipartite}
 
For every positive integer $N$, we have $\delta_q(T_{q+1}(N))=N$.
In each of the two cases below, let $V_0,V_1,\ldots,V_q$ denote the parts of
the complete $(q+1)$-partite graph under consideration.

We use the following observation in both cases. Suppose that the colors are
labeled by $0,1,\ldots,q$ and that every vertex in $V_j$ has a color in its
list labeled $j$. Color each vertex with one such color. A color has only one
label, so vertices in different parts cannot receive the same color. Thus the
coloring is proper.

We first consider the general partition $\lambda$. Let $L$ be an arbitrary
$\lambda$-assignment of $T_{q+1}(N_1)$, and label every color independently
and uniformly from $\{0,1,\ldots,q\}$. For a vertex $v\in V_j$, the
probability that $L(v)$ contains no color labeled $j$ is
$(q/(q+1))^{|\lambda|}$.
Hence the expected number of vertices having no color labeled with the index
of their part is
\[
 (q+1)N_1\left(\frac{q}{q+1}\right)^{|\lambda|}<1.
\]
Some labeling leaves no such vertex. By the preceding observation, we obtain
a proper $L$-coloring. Thus $T_{q+1}(N_1)$ is $\lambda$-choosable.

To obtain the stated lower bound for $f(\lambda)$, take
\[
 N_1=
 \left\lceil
 \frac{1}{q+1}\left(1+\frac1q\right)^{|\lambda|}
 \right\rceil-1
\]
whenever this integer is positive. This value satisfies the strict inequality
in the proposition, and the graph above has $q$-codegree $N_1$. Hence
\[
 f(\lambda)\geq N_1+1
 =\left\lceil
 \frac{1}{q+1}\left(1+\frac1q\right)^{|\lambda|}
 \right\rceil.
\]
If $N_1=0$, the right-hand side is $1$. In this case the same conclusion
follows from the $\lambda$-choosable graph $K_q$, for which
$\delta_q(K_q)=0$.

We now assume that every part of the partition is equal to $k$. Let $L$ be an
arbitrary $\equalparts{k}{q}$-assignment of $T_{q+1}(N_2)$, with color pools
$X_1,\ldots,X_q$. For every $i\in[q]$, independently assign each color in
$X_i$ label $0$ with probability $\rho_q$ and label $i$ with probability
$1-\rho_q$. If $v\in V_0$, the probability that $L(v)$ contains no color
labeled $0$ is $(1-\rho_q)^{qk}=\rho_q^k$.
If $v\in V_i$ for some $i\in[q]$, only the $k$ colors of $L_i(v)$  
may receive label $i$, and the probability that none of them does is
$\rho_q^k$.
The expected number of vertices having no color labeled with the index of
their part is therefore $(q+1)N_2\rho_q^k<1$.
Some labeling again leaves no such vertex. By the preceding observation,
we obtain a proper $L$-coloring, so $T_{q+1}(N_2)$ is
$\equalparts{k}{q}$-choosable.

Finally, take $N_2=\left\lceil \rho_q^{-k}/(q+1)\right\rceil-1$
whenever this integer is positive. Then $(q+1)N_2\rho_q^k<1$, and therefore
\[
 f(\equalparts{k}{q})\geq N_2+1
 =\left\lceil\frac{\rho_q^{-k}}{q+1}\right\rceil.
\]
If $N_2=0$, the right-hand side is again $1$, and the conclusion follows from
$K_q$.

\section{Proof of Theorem~\ref{thm:existence}}
\label{sec:surrounded}

Let $\lambda=\{k_1,k_2,\ldots,k_q\}$ be a fixed partition. Let $G$ be a
graph, and let $d=\delta_q(G)$.
Let $X_1,\ldots,X_q$ be pairwise disjoint color pools with $|X_i| = k_i^2$ for every $i \in [q]$.

The proof uses a deferred-exposure form of Alon's random-list
argument~\cite{Alon2000}. The ordinary list-coloring argument works as
follows. All lists use a fixed color palette and are chosen independently at
random at the beginning, but are revealed only when needed. The goal is to
find a list assignment $L$ for which $G$ has no $L$-coloring. It is enough
to show that this occurs with positive probability. We first reveal
the lists on a sparse set $B$. Many vertices outside $B$ are then
\emph{good}, each of which has the property that for every $L$-coloring of
$G[B]$, the colors used by its $B$-neighbors meet every half-set of the
palette. These colors therefore contain more than half of the palette. When
the previously unexposed random list at a good vertex is finally revealed,
it has a substantial chance of containing only colors already used by its
$B$-neighbors. We complete the argument with a union bound over the
$L$-colorings of $G[B]$.

For a general partition $\lambda=\{k_1,\ldots,k_q\}$, the color palette is
divided into $q$ pools. To play the same role as a good vertex in the
ordinary list-coloring argument, a vertex must satisfy the half-set property
in every pool for every $L$-coloring of $G[B]$. It is difficult to establish
all these properties by revealing $B$ in one step. We therefore divide $B$
into layers and reveal them in stages. By revealing the layers in this order,
we establish the required half-set properties recursively. Before describing the random experiment that chooses the lists and assigns
vertices to the layers, we introduce the notation and parameters used in the
construction.

For $i \in [q]$ and $\varnothing \ne I \subseteq [q]$, let
\[
 \alpha_i =
 \frac{\binom{\lfloor |X_i|/2\rfloor + 1}{k_i}}
 {\binom{|X_i|}{k_i}}=\frac{\binom{\lfloor k_i^2/2\rfloor + 1}{k_i}}
 {\binom{k_i^2}{k_i}},
 \quad
 \rho_I =
 \prod_{i \in I}
 \frac{\binom{\lceil |X_i|/2\rceil}{k_i}}
 {\binom{|X_i|}{k_i}}=\prod_{i \in I}
 \frac{\binom{\lceil k_i^2 /2\rceil}{k_i}}
 {\binom{k_i^2}{k_i}}.
\]
Let $c_{[q]}=1$. For $\varnothing\ne I\subsetneq[q]$, define $c_I$ and
$\hat{c}_I$ recursively by
\begin{equation}\label{eq:c-recursion}
 \hat{c}_{I} = \sum_{J \supsetneq I}c_J,
 \quad
 c_I = \frac{12\hat{c}_{I}\ln |\lambda|}
 {\rho_I\prod_{j \notin I}\alpha_j}.
\end{equation}
Finally, let
$\widetilde{c} = \sum_{\varnothing \ne J \subseteq [q]}c_J$.

With these parameters fixed, assume for the moment that
$\widetilde{c}/\sqrt{d} \leq 1/8$.
The random experiment consists of the following two independent families of
choices.
\begin{enumerate}
\item For every vertex $v$ and every $i \in [q]$, choose $L_i(v)$ uniformly
from $\binom{X_i}{k_i}$, independently over all pairs $(v,i)$, and define
$L(v)=\bigcup_{i=1}^q L_i(v)$.
\item Independently of all list components, assign each vertex $v$ a label
$I\subseteq[q]$, independently over the vertices. Let $B_I$ be the set of
vertices with label $I$. For every nonempty $I\subseteq[q]$, let
$\Prob(v\in B_I)=c_I/\sqrt d$, and let
$\Prob(v\in B_\varnothing)=1-\widetilde c/\sqrt d$.
\end{enumerate}
For every $I \subseteq [q]$, define
\[
 \hat{B}_{I} = \bigcup_{J \supsetneq I}B_J.
\]
Thus $\hat{B}_{[q]}=\varnothing$, while $\hat{B}_{\varnothing}$ is the union
of all nonempty layers. Let $B=\hat{B}_{\varnothing}$, as in Alon's
argument~\cite{Alon2000}. Since the sets $B_I$ form a partition of $V(G)$,
we have $B_\varnothing=V(G)\setminus B$.

\subsection{Blocking events and surrounded vertices}\label{subsec:blocking-events}

A \emph{state} is a triple $(S,I,\mathbf{P}_I)$ satisfying
$S \subseteq V(G)$,
$I \subseteq [q]$, $|S| = |I|$, and
$\mathbf{P}_I = (P_i)_{i \in I}$, where
$P_i\in\binom{X_i}{\lceil |X_i|/2\rceil}$ for every $i\in I$.
The empty state has $S = I = \varnothing$ and the unique empty tuple
$\mathbf{P}_\varnothing$.
For $I \subseteq [q]$, a proper $L$-coloring $\phi$ of
$G[\hat{B}_{I} \cap N(S)]$, a vertex $u \in N(S)$, and $j \notin I$, define
\[
 \Use_j^\phi(u, S, I) =
 \phi\bigl(\hat{B}_{I} \cap N(S) \cap N(u)\bigr) \cap X_j.
\]
We say that $\phi$ \emph{avoids} $\mathbf{P}_I$ if it uses no color in
$\bigcup_{i \in I}P_i$.

We first define the obstruction needed at a fixed state.

\begin{definition}\label{def:blocking-events}
For a state $(S,I,\mathbf{P}_I)$ with $I \ne \varnothing$, let
$\mathcal{E}(S,I,\mathbf{P}_I)$ be the event that, for every proper
$L$-coloring $\phi$ of
$G[\hat{B}_{I} \cap N(S)]$ that avoids $\mathbf{P}_I$, there is a vertex
$u \in B_I \cap N(S)$ satisfying
\begin{enumerate}[label=(\alph*),ref=(\alph*)]
\item\label{item:event-prescribed}
$L_i(u) \subseteq P_i$ for every $i \in I$,
\item\label{item:event-blocking}
$L_j(u) \subseteq \Use_j^\phi(u, S, I)$ for every $j \notin I$.
\end{enumerate}
We call the state $(S,I,\mathbf{P}_I)$ \emph{blocking} if
$\mathcal{E}(S,I,\mathbf{P}_I)$ occurs.
\end{definition}

When $I = [q]$, Definition~\ref{def:blocking-events} has the following
particularly simple form. Since $\hat{B}_{[q]} = \varnothing$, the graph
$G[\hat{B}_{[q]} \cap N(S)]$ has the unique empty coloring $\phi$, which
avoids $\mathbf{P}_{[q]}$. Moreover, there is no $j \notin [q]$, so
Condition~\ref{item:event-blocking} is vacuous. Hence
$\mathcal{E}(S,[q],\mathbf{P}_{[q]})$ is precisely the event that there is a
vertex $w \in B_{[q]} \cap N(S)$ such that
$L_i(w) \subseteq P_i$ for every $i \in [q]$.

The role of the blocking event is to prevent a proper $L$-coloring $\phi$
of $G[\hat{B}_I\cap N(S)]$ that avoids $\mathbf{P}_I$ from being extended
to the corresponding vertex $u$ while still avoiding $\mathbf{P}_I$.
Condition~\ref{item:event-blocking} prevents $u$ from using any list
component $L_j(u)$ with $j\notin I$, since every color in $L_j(u)$ is
already used by a neighbor of $u$ under $\phi$. Thus every proper extension
of $\phi$ to $u$ must use a color from $L_i(u)$ for some $i\in I$. By
Condition~\ref{item:event-prescribed}, $L_i(u)\subseteq P_i$. The extension
therefore uses a color in $\bigcup_{i\in I}P_i$ and no longer avoids
$\mathbf{P}_I$.
 
To find the vertex $u$ required in
Definition~\ref{def:blocking-events}, we need the following definition,
where a member of
$\binom{X_i}{\lceil |X_i|/2\rceil}$ is called a \emph{half-set} of $X_i$
for every $i\in[q]$.

\begin{definition}\label{def:surrounded}
Let $(S,I,\mathbf{P}_I)$ be a state with $I \subsetneq [q]$. A vertex
$u \in N(S) \setminus \hat{B}_{I}$ is \emph{surrounded at}
$(S,I,\mathbf{P}_I)$ if, for every $j \notin I$ and every half-set $P_j$ of
$X_j$, the blocking event
$\mathcal{E}(S \cup \{u\},J,\mathbf{P}_J)$ occurs, where
$J = I \cup \{j\}$ and $\mathbf{P}_J$ is the tuple extending $\mathbf{P}_I$ with
$j$-coordinate $P_j$.
\end{definition}

The role of surrounded vertices is captured by the following lemma. In every
remaining color pool, the neighbors of a surrounded vertex use more than half
of the colors.

\begin{lemma}\label{lem:recursive-hit}
Suppose that $u$ is surrounded at
$(S,I,\mathbf{P}_I)$, where $I$ is allowed to be empty, and let $\phi$ be a
proper $L$-coloring of $G[\hat{B}_{I} \cap N(S)]$ that avoids
$\mathbf{P}_I$.
For every $j \notin I$, the set $\Use_j^\phi(u, S, I)$ meets every
half-set of $X_j$. Consequently,
\[
 |\Use_j^\phi(u, S, I)| \geq \lfloor k_j^2/2\rfloor + 1.
\]
\end{lemma}

\begin{proof}
Fix $j \notin I$ and suppose that some half-set $P_j$ of $X_j$ is disjoint
from $\Use_j^\phi(u, S, I)$. Let $J = I \cup \{j\}$, and let $\mathbf{P}_J$
extend $\mathbf{P}_I$ by the $j$-coordinate $P_j$.
Since $u$ is surrounded at $(S,I,\mathbf{P}_I)$, the blocking event
$\mathcal{E}(S \cup \{u\},J,\mathbf{P}_J)$ occurs.

Let $\phi'$ be the restriction
of $\phi$ to $\hat{B}_{J} \cap N(S \cup \{u\})$. Since $J \supsetneq I$, we have
$\hat{B}_{J} \cap N(S \cup \{u\})
 \subseteq \hat{B}_{I} \cap N(S) \cap N(u)$. Thus $\phi'$ is a proper
$L$-coloring of $G[\hat{B}_{J} \cap N(S \cup \{u\})]$ and avoids
$\mathbf{P}_I$. We claim that $\phi'$ also uses no color in $P_j$. If
$\phi'(v) \in P_j$ for some
$v \in \hat{B}_{J} \cap N(S \cup \{u\})$, then
$v \in \hat{B}_{I} \cap N(S) \cap N(u)$ and $\phi'(v) = \phi(v)$. Since
$P_j \subseteq X_j$, by the definition of $\Use_j^\phi(u, S, I)$, we have
$\phi'(v) \in \Use_j^\phi(u, S, I)$. Hence
$\phi'(v) \in P_j \cap \Use_j^\phi(u, S, I)$, contrary to the choice of $P_j$.
Therefore, $\phi'$ avoids $\mathbf{P}_J$.

Since the blocking event $\mathcal{E}(S \cup \{u\},J,\mathbf{P}_J)$ occurs and
$\phi'$ avoids $\mathbf{P}_J$, there is a vertex
$z \in B_J \cap N(S \cup \{u\})$ by
Definition~\ref{def:blocking-events}. In particular,
$z \in \hat{B}_{I} \cap N(S) \cap N(u)$, so $\phi(z)$ is defined. Choose
$r \in [q]$ such that $\phi(z) \in L_r(z)$. If $r \in I$, then by
Definition~\ref{def:blocking-events}\ref{item:event-prescribed}, we have
$\phi(z) \in P_r$, contrary to the fact that $\phi$ avoids $\mathbf{P}_I$.
If $r = j$, then by
Definition~\ref{def:blocking-events}\ref{item:event-prescribed} and the definition of
$\Use_j^\phi(u, S, I)$, we have
$\phi(z) \in P_j \cap \Use_j^\phi(u, S, I)$, contrary to the choice of $P_j$.
Finally, if $r \notin J$, then by
Definition~\ref{def:blocking-events}\ref{item:event-blocking}, we have
$\phi(z) \in \Use_r^{\phi'}(z, S \cup \{u\}, J)$. Hence there is a vertex
$v \in \hat{B}_{J} \cap N(S \cup \{u\}) \cap N(z)$ such that
$\phi(v) = \phi'(v) = \phi(z)$, contradicting the fact that $\phi$ is a proper coloring on $G[\hat{B}_{I} \cap N(S)]$.

Thus $\Use_j^\phi(u, S, I)$ meets every
half-set of $X_j$. Its complement has size at most
$\lceil k_j^2/2\rceil - 1$, and hence
$|\Use_j^\phi(u, S, I)| \geq \lfloor k_j^2/2\rfloor + 1$.
\end{proof}

\subsection{The local blocking lemma}\label{subsec:local-blocking}

The blocking events were defined in the previous subsection. We now prove that,
for every state $(S,I,\mathbf{P}_I)$ with $I \ne \varnothing$, the probability
that the blocking event $\mathcal{E}(S,I,\mathbf{P}_I)$ does not occur is at most
$\varepsilon_I(d)$, where the error functions $\varepsilon_I(d)$ are defined
recursively as follows. First, let
\[
 \varepsilon_{[q]}(d) = \exp(-\rho_{[q]}\sqrt{d}).
\]
For $\varnothing \ne I \subsetneq [q]$, once $\varepsilon_J(d)$ has been
defined for every $J \supsetneq I$, let
\begin{align*}
 \varepsilon_I(d)
=  \exp\left(-\frac{\hat{c}_{I}\sqrt{d}}{3}\right)
 +4\sum_{j \notin I}
 \binom{k_j^2}{\lceil k_j^2/2\rceil}
 \varepsilon_{I \cup \{j\}}(d) 
 +\exp\left(-\frac{c_I\rho_I\sqrt{d}}{16}\right)
 +\exp\left(-\hat{c}_{I}\sqrt{d}\ln |\lambda|\right).
\end{align*}
For every nonempty $I \subseteq [q]$, by induction on $q - |I|$, we have
$\varepsilon_I(d) \to 0$ as $d \to \infty$.

Before stating and proving the lemma, we introduce the following standard form of the Chernoff bound.

\begin{lemma}\label{lem:chernoff}
	Let $X \sim \Bin(n,p)$ and $\mu=np$. Then 
	\begin{enumerate}[label=(\arabic*),ref=(\arabic*)]
		\item\label{item:chernoff-lower}
		$\Prob(X < \mu/2) \leq e^{-\mu/8}$.
		\item\label{item:chernoff-upper}
		$\Prob(X > 2\mu) \leq e^{-\mu/3}$.
	\end{enumerate}
\end{lemma}

The desired probability bound is as follows.

\begin{lemma}
\label{lem:local-blocking}
Suppose that $d = \delta_q(G)$ satisfies $\widetilde{c}/\sqrt{d} \leq 1/8$.
For every state $(S,I,\mathbf{P}_I)$ with $\varnothing \ne I \subseteq [q]$, the probability
that the blocking event $\mathcal{E}(S,I,\mathbf{P}_I)$ does not occur is at most
$\varepsilon_I(d)$.
\end{lemma}

\begin{proof}
We use induction on $q - |I|$. By the definition of $d = \delta_q(G)$, we have
$|N(S)| \geq d$.

If $I = [q]$, then a fixed vertex $w \in N(S)$ belongs to $B_{[q]}$ and
satisfies $L_i(w) \subseteq P_i$ for every $i \in [q]$ with probability
$\rho_{[q]}/\sqrt d$. These events are independent for distinct vertices.
Hence the probability that the blocking event
$\mathcal{E}(S,[q],\mathbf{P}_{[q]})$ does not occur is at most
\[
 \left(1 - \frac{\rho_{[q]}}{\sqrt d}\right)^{|N(S)|}
 \leq \exp\left(-\frac{\rho_{[q]}|N(S)|}{\sqrt d}\right)
 \leq \exp(-\rho_{[q]}\sqrt{d})
 = \varepsilon_{[q]}(d).
\]

Now let $\varnothing \ne I \subsetneq [q]$ and suppose that the lemma holds
for every $J \supsetneq I$.
We divide the induction step into three steps. In Step~1, we reveal the sets
$B_J$, for $J \supsetneq I$, and all list components of the vertices in
$\hat{B}_{I}$, and show that many vertices
$u \in N(S) \setminus \hat{B}_{I}$ are surrounded at $(S,I,\mathbf{P}_I)$. In
Step~2, we reveal $B_I$ and the components $L_i(u)$ for $u\in B_I$ and
$i\in I$. We show that many of the surrounded vertices have
$L_i(u) \subseteq P_i$ for every $i \in I$. These are the candidate blocking vertices
for the event
$\mathcal{E}(S,I,\mathbf{P}_I)$. In Step~3, we reveal their remaining list
components. For every proper $L$-coloring $\phi$ of
$G[\hat{B}_{I} \cap N(S)]$ that avoids $\mathbf{P}_I$, we show that some
candidate vertex satisfies Definition~\ref{def:blocking-events}%
\ref{item:event-blocking}. Hence the blocking event
$\mathcal{E}(S,I,\mathbf{P}_I)$ occurs.

\medskip
\noindent
\textbf{Step 1.} Reveal the sets $B_J$, $J \supsetneq I$, and the lists on
$\hat{B}_{I}$.\par\nobreak
\medskip

In this step, we reveal the sets $B_J$ for $J \supsetneq I$,
together with all list components of the vertices in their union
$\hat{B}_{I}$. After these choices are revealed, we know whether each blocking
event whose second coordinate strictly contains $I$ occurs. We therefore also
know whether each $u \in N(S) \setminus \hat{B}_{I}$ is surrounded at
$(S,I,\mathbf{P}_I)$, and this decision does not use whether $u \in B_I$ or any
list component of $u$.

Since the sets $B_J$ are pairwise disjoint, each vertex belongs to
$\hat{B}_{I}$ independently with probability $\hat{c}_{I}/\sqrt{d}$.
Hence $|\hat{B}_{I} \cap N(S)|$ has binomial distribution with mean
$\hat{c}_{I}|N(S)|/\sqrt{d}$.
By Lemma~\ref{lem:chernoff}\ref{item:chernoff-upper}, we have
\begin{equation}\label{eq:upper-tail}
 \Prob\left(
 |\hat{B}_{I} \cap N(S)| > \frac{2\hat{c}_{I}|N(S)|}{\sqrt{d}}
 \right)
 \leq \exp\left(-\frac{\hat{c}_{I}|N(S)|}{3\sqrt{d}}\right)
 \leq \exp\left(-\frac{\hat{c}_{I}\sqrt{d}}{3}\right).
\end{equation}

For each $j \notin I$ and each half-set $P_j$ of $X_j$, let
$J = I \cup \{j\}$, and let $\mathbf{P}_J$ be the tuple extending $\mathbf{P}_I$
with $j$-coordinate $P_j$. Let $F$ be the set of vertices
$u \in N(S) \setminus \hat{B}_{I}$ for which the blocking event
$\mathcal{E}(S \cup \{u\},J,\mathbf{P}_J)$ fails for at least one such pair
$(j,P_j)$. Thus $F$ is determined by the choices revealed in Step~1.
For a fixed $u \in N(S)$, by the induction hypothesis and the union bound over
all such pairs $(j, P_j)$, we have
\[
 \Prob(u \in F) \leq
 \sum_{j \notin I}
 \binom{k_j^2}{\lceil k_j^2/2\rceil}
 \varepsilon_{I \cup \{j\}}(d).
\]
By linearity of expectation and Markov's inequality, we have
\begin{equation}\label{eq:failed-vertex-tail}
 \Prob\left(|F| > \frac{|N(S)|}{4}\right)
 \leq 4\sum_{j \notin I}
 \binom{k_j^2}{\lceil k_j^2/2\rceil}
 \varepsilon_{I \cup \{j\}}(d).
\end{equation}
The events that different vertices belong to $F$ need not be independent,
since only linearity of expectation and Markov's inequality are used.

\medskip
\noindent
\textbf{Step 2.} Reveal $B_I$, together with $L_i(u)$ for every $u\in B_I$
and $i\in I$.
\par\nobreak
\medskip

Condition on any possible result of Step~1 for which
\begin{equation}\label{eq:first-step-bounds}
 |\hat{B}_{I} \cap N(S)|
 \leq \frac{2\hat{c}_{I}|N(S)|}{\sqrt{d}},
 \quad
 |F| \leq \frac{|N(S)|}{4}.
\end{equation}
Let $W=N(S)\setminus(\hat{B}_I\cup F)$. By
Definition~\ref{def:surrounded}, this is precisely the set of vertices in
$N(S)\setminus\hat{B}_I$ that are surrounded at
$(S,I,\mathbf{P}_I)$. Since $\hat{c}_I\leq\widetilde{c}$ and
$\widetilde{c}/\sqrt{d}\leq 1/8$, we have
\[
 |W| \geq |N(S)| - \frac{2\hat{c}_{I}|N(S)|}{\sqrt{d}}
 - \frac{|N(S)|}{4} \geq \frac{|N(S)|}{2}.
\]
For $u \in W$, let $\mathcal{A}_u$ be the event that $u \in B_I$ and
$L_i(u) \subseteq P_i$ for every $i \in I$.
For every $u \in W$, the information revealed in Step~1 includes
$u \notin \hat{B}_{I}$. Moreover, since the sets $B_J$ are pairwise
disjoint, $\mathcal{A}_u$ implies $u \notin \hat{B}_{I}$. In the original
random experiment,
\[
 \Prob(\mathcal{A}_u) = \frac{c_I\rho_I}{\sqrt{d}},
 \quad
 \Prob(u \notin \hat{B}_{I})
 = 1 - \frac{\hat{c}_{I}}{\sqrt{d}}.
\]
It follows that
\[
\begin{aligned}
 \Prob\bigl(\mathcal{A}_u \mid u \notin \hat{B}_{I}\bigr) =
 \frac{
 \Prob\bigl(\mathcal{A}_u
 \cap \{u \notin \hat{B}_{I}\}\bigr)}
 {\Prob(u \notin \hat{B}_{I})} =
 \frac{\Prob(\mathcal{A}_u)}
 {\Prob(u \notin \hat{B}_{I})} =
 \frac{c_I\rho_I}{\sqrt{d} - \hat{c}_{I}}
 \geq \frac{c_I\rho_I}{\sqrt{d}}.
\end{aligned}
\]
For $u\in W$, Step~1 reveals that $u\notin\hat{B}_I$, but it does not
reveal whether $u\in B_I$ or any list component of $u$. Conditional on this
result of Step~1, the events $\mathcal A_u$, $u\in W$, remain independent,
and by the preceding
calculation each occurs with probability
$c_I\rho_I/(\sqrt d-\hat c_I)$.

Let $R$ be the set of vertices $u\in W$ for which $\mathcal A_u$ occurs.
Then $|R|$ has a binomial distribution with parameters $|W|$ and
$c_I\rho_I/(\sqrt d-\hat c_I)$. Thus
\[
\Ex(|R|)
=\frac{c_I\rho_I|W|}{\sqrt d-\hat c_I}
\geq\frac{c_I\rho_I|N(S)|}{2\sqrt d}.
\]
By Lemma~\ref{lem:chernoff}\ref{item:chernoff-lower},
\begin{equation}\label{eq:candidate-tail}
 \Prob\left(
 |R| < \frac{c_I\rho_I|N(S)|}{4\sqrt{d}}
 \right)
 \leq \exp\left(-\frac{c_I\rho_I|N(S)|}{16\sqrt{d}}\right)
 \leq \exp\left(-\frac{c_I\rho_I\sqrt{d}}{16}\right).
\end{equation}

\medskip
\noindent
\textbf{Step 3.} Reveal the remaining list components of the vertices in $R$.
\par\nobreak
\medskip

Condition further on the event that
\begin{equation}\label{eq:candidate-count}
 |R| \geq \frac{c_I\rho_I|N(S)|}{4\sqrt{d}}.
\end{equation}
At this point, the set $B_I$, the components $L_i(u)$ for $u\in B_I$ and
$i\in I$, and the set $R$ have been revealed.
Every vertex of $R$ lies in $B_I \cap N(S)$ and satisfies
Definition~\ref{def:blocking-events}\ref{item:event-prescribed}. It is also
surrounded at $(S,I,\mathbf{P}_I)$ by the definition of $W$. Reveal
$L_j(u)$ for every $u \in R$ and $j \notin I$.

After Step~1, the lists on $\hat{B}_{I} \cap N(S)$ are fixed. The number of
proper $L$-colorings of
$G[\hat{B}_{I} \cap N(S)]$ that avoid $\mathbf{P}_I$ is at most
\begin{equation}\label{eq:number-upper-colorings}
 |\lambda|^{|\hat{B}_{I} \cap N(S)|}
 \leq |\lambda|^{2\hat{c}_{I}|N(S)|/\sqrt{d}}
 = \exp\left(
 \frac{2\hat{c}_{I}|N(S)|\ln |\lambda|}{\sqrt{d}}
 \right).
\end{equation}
Fix such a coloring $\phi$. Lemma~\ref{lem:recursive-hit} shows that
$|\Use_j^\phi(u,S,I)|\geq\lfloor k_j^2/2\rfloor+1$ for every $u\in R$ and
$j\notin I$. By
the definition of $\alpha_j$, a uniformly random member of
$\binom{X_j}{k_j}$ is contained in $\Use_j^\phi(u, S, I)$ with probability
at least $\alpha_j$. For the list components revealed in Step~3, a fixed
vertex $u \in R$ therefore satisfies Condition~\ref{item:event-blocking} in
Definition~\ref{def:blocking-events} with probability at least
$\prod_{j \notin I}\alpha_j$. These events are independent over the vertices
of $R$. By Inequality~\eqref{eq:candidate-count}, the probability that no vertex of $R$
satisfies Condition~\ref{item:event-blocking} in
Definition~\ref{def:blocking-events} is at most
\[
 \left(1 - \prod_{j \notin I}\alpha_j\right)^{|R|}
 \leq \exp\left(-|R|\prod_{j \notin I}\alpha_j\right)
 \leq \exp\left(
 -\frac{c_I\rho_I|N(S)|}{4\sqrt{d}}
 \prod_{j \notin I}\alpha_j
 \right).
\]

Under the conditioning in \eqref{eq:first-step-bounds}
and~\eqref{eq:candidate-count}, we obtain from
Equality~\eqref{eq:c-recursion}, Inequality~\eqref{eq:number-upper-colorings},
and the union bound that the probability that some such coloring has no
vertex in $R$ satisfying
Definition~\ref{def:blocking-events}\ref{item:event-blocking} is at most
\begin{equation}\label{eq:third-step-failure}
\begin{aligned}
|\lambda|^{|\hat{B}_{I} \cap N(S)|}
\left(1 - \prod_{j \notin I}\alpha_j\right)^{|R|}
 & \leq \exp\left(
 \frac{2\hat{c}_{I}|N(S)|\ln |\lambda|}{\sqrt{d}}
 -\frac{c_I\rho_I|N(S)|}{4\sqrt{d}}
 \prod_{j \notin I}\alpha_j
 \right) \\
 & = \exp\left(
 -\frac{\hat{c}_{I}|N(S)|\ln |\lambda|}{\sqrt{d}}
 \right)\\
 & \leq
 \exp\left(-\hat{c}_{I}\sqrt{d}\ln |\lambda|\right).
\end{aligned}
\end{equation}

By construction, $R \subseteq B_I \cap N(S)$, and every vertex of $R$
satisfies Definition~\ref{def:blocking-events}\ref{item:event-prescribed}.
If Inequalities~\eqref{eq:first-step-bounds} and~\eqref{eq:candidate-count}
hold and the event estimated in Inequality~\eqref{eq:third-step-failure} does
not occur, then every coloring in Definition~\ref{def:blocking-events} has a
vertex in $R$ satisfying
Definition~\ref{def:blocking-events}\ref{item:event-blocking}. Thus the
blocking event $\mathcal{E}(S,I,\mathbf{P}_I)$ occurs. Since the estimates in
Steps~2 and~3 apply after conditioning on any earlier information satisfying
the required bounds,
by Inequalities~\eqref{eq:upper-tail}, \eqref{eq:failed-vertex-tail},
\eqref{eq:candidate-tail}, and~\eqref{eq:third-step-failure}, the probability
that this blocking event does not occur is at most
\begin{align*}
 \exp\left(-\frac{\hat{c}_{I}\sqrt{d}}{3}\right)
 +4\sum_{j \notin I}
 \binom{k_j^2}{\lceil k_j^2/2\rceil}
 \varepsilon_{I \cup \{j\}}(d)
 +\exp\left(-\frac{c_I\rho_I\sqrt{d}}{16}\right)
 +\exp\left(-\hat{c}_{I}\sqrt{d}\ln |\lambda|\right)
 = \varepsilon_I(d).
\end{align*}
This completes the induction.
\end{proof}

\subsection{The final exposure}
\label{subsec:final-exposure}

We now use the local estimate to select suitable values for the information
exposed on the nonempty layers. Let $A$ be the set of vertices in
$V(G) \setminus B$ that are
surrounded at the empty state. Thus $a \in A$ precisely when the
blocking event $\mathcal{E}(\{a\},\{i\},\mathbf{P}_{\{i\}})$ occurs for every
$i \in [q]$ and every half-set $P_i$ of $X_i$, where
$\mathbf{P}_{\{i\}} = (P_i)$. By Lemma~\ref{lem:local-blocking} and the union
bound, we have
\begin{equation}\label{eq:bad-a}
 \Prob(a \notin A)
 \leq \frac{\widetilde{c}}{\sqrt{d}}
 +\sum_{i = 1}^q
 \binom{k_i^2}{\lceil k_i^2/2\rceil}
 \varepsilon_{\{i\}}(d).
\end{equation}

Choose $D_\lambda$ large enough so that for every $d \geq D_\lambda$,
\begin{enumerate}[label=(D\arabic*),ref=(D\arabic*),labelindent=2em,
	leftmargin=*, itemsep=0.5em,topsep=0.7em]
\item\label{cond:final-error} $\widetilde{c}/\sqrt{d}
+\sum_i\binom{k_i^2}{\lceil k_i^2/2\rceil}
\varepsilon_{\{i\}}(d) \leq 1/16$.
\item\label{cond:final-union} $16\widetilde{c}\ln |\lambda|/\sqrt{d}
 \leq \prod_{i = 1}^q \alpha_i$.
\end{enumerate}
By condition~\ref{cond:final-error}, we have
$\widetilde{c}/\sqrt{d} \leq 1/8$. Such a choice exists
because all constants depend only on $\lambda$ and all error functions tend
to zero.

Let $d = \delta_q(G) \geq D_\lambda$ and let $n = |V(G)|$. By
Inequality~\eqref{eq:bad-a}, the expected number of vertices outside $A$ is at
most $n/16$, while $\Ex(|B|) = \widetilde{c}n/\sqrt{d}$.
By Markov's inequality, we have
\[
 \Prob(|A| < n/2) \leq \frac{1}{8},
 \quad
 \Prob\left(|B| > \frac{4\widetilde{c}n}{\sqrt{d}}\right) \leq \frac{1}{4}.
\]
Consequently, with positive probability the information exposed so far satisfies
\begin{equation}\label{eq:AB-size}
 |A| \geq n/2,
 \quad
 |B| \leq \frac{4\widetilde{c}n}{\sqrt{d}}.
\end{equation}
Condition on any realization of the exposed information satisfying
\eqref{eq:AB-size}. The lists on
$V(G)\setminus B$ remain unexposed.

Let $\Phi_B$ be the set of proper $L$-colorings of $G[B]$.
If $\Phi_B = \varnothing$, then the full random assignment is already
uncolorable, regardless of the unexposed lists. We may therefore assume that
$\Phi_B \ne \varnothing$.

We now reveal the lists on $A$. Under this conditioning,
the components $L_i(a)$, for $a\in A$ and $i\in[q]$, remain mutually
independent and uniformly distributed in $\binom{X_i}{k_i}$. Indeed, whether
$a$ belongs to $A$ was determined without exposing any component of $L(a)$.

Fix $\phi \in \Phi_B$ and $a \in A$. For every $i \in [q]$, since $a$ is
surrounded at the empty state, by Lemma~\ref{lem:recursive-hit} with
$S = I = \varnothing$, we have
\[
 |\Use_i^\phi(a, \varnothing, \varnothing)|
 \geq \lfloor |X_i|/2\rfloor + 1 = \lfloor k_i^2/2\rfloor + 1.
\]
Hence, by the definition of $\alpha_i$ and independence over $i \in [q]$,
the probability that
$L_i(a) \subseteq \Use_i^\phi(a, \varnothing, \varnothing)$ for every
$i \in [q]$ is at least $\prod_{i = 1}^q \alpha_i$.
On this event, no color in $L(a)$ is available for extending $\phi$ to $a$.
By independence over $a \in A$, the probability that $\phi$ extends to every
vertex of $A$ is at most
\[
 \left(1 - \prod_{i = 1}^q \alpha_i\right)^{|A|}
 \leq \exp\left(-|A|\prod_{i = 1}^q \alpha_i\right).
\]
Moreover, $|\Phi_B| \leq |\lambda|^{|B|}$. By the inequalities
in~\eqref{eq:AB-size}, condition~\ref{cond:final-union}, and the union bound over all
$\phi \in \Phi_B$, the probability that some $\phi$ extends to every vertex of
$A$ is at most
\begin{align*}
 \exp\bigl(|B|\ln |\lambda|
           -|A|\prod_{i = 1}^q \alpha_i\bigr) \leq
 \exp\left(
 \frac{4\widetilde{c}n\ln |\lambda|}{\sqrt{d}}
 -\frac{n}{2}\prod_{i = 1}^q \alpha_i\right)
  \leq \exp\left(-\frac{n}{4}\prod_{i = 1}^q \alpha_i\right) < 1.
\end{align*}
We can therefore choose the previously unexposed lists on $A$ so that no
coloring of $G[B]$ extends to all of $A$. The lists already sampled on
$V(G) \setminus (A \cup B)$ are irrelevant. Hence some outcome of the original
single random experiment is an uncolorable $\lambda$-assignment. Since $G$
was arbitrary with $\delta_q(G)\geq D_\lambda$, it follows that
$f(\lambda)\leq D_\lambda<\infty$. We have proved the first assertion of
Theorem~\ref{thm:existence}.

\subsection{The asymptotic upper bound}
\label{subsec:threshold}

It remains to prove the asymptotic upper bound in
Theorem~\ref{thm:existence}. Throughout this subsection, $q$ is fixed, and
every $o(1)$ term tends to zero as $|\lambda|\to\infty$. We first estimate
the probabilities $\alpha_i$ and $\rho_I$, the recursive
constants $c_I$, and the error terms $\varepsilon_I(d)$.

The following elementary estimate is the binomial-ratio bound used in
Alon's proof~\cite{Alon2000}. We include its short proof for completeness.

\begin{lemma}\label{lem:half-binomial-ratio}
	For every positive integer $k$ and every integer
	$s\geq\lceil k^2/2\rceil$, $\binom{s}{k}/\binom{k^2}{k} \geq 2^{-(k+1)}$.
\end{lemma} 
\begin{proof}
	The result is immediate for $k=1$, so assume that $k\geq2$. Since $s\geq k^2/2$, we have
	\[ 
	\frac{\binom{s}{k}}{\binom{k^2}{k}} =  \prod_{i=0}^{k-1}\frac{s-i}{k^2-i} \geq \frac{1}{2^k}\prod_{i=0}^{k-1}\frac{k^2-2i}{k^2-i} =
	\frac{1}{2^k}\prod_{i=0}^{k-1}\left(1-\frac{i}{k^2-i}\right) \geq  \frac{1}{2^k}\left(1-\sum_{i=0}^{k-1}\frac{i}{k^2-k}\right)  =2^{-(k+1)}. 
	\]
	The second inequality uses $\prod_i(1-x_i)\geq1-\sum_i x_i$ for $0\leq x_i\leq1$ and
	$k^2-i\geq k^2-k$.
\end{proof} 

By Lemma~\ref{lem:half-binomial-ratio}, applied to the ratios defining
$\alpha_i$ and $\rho_I$, we have
\begin{equation}\label{eq:alpha-rho-lower}
 \alpha_i\geq2^{-(k_i+1)},
 \quad
 \rho_I\geq2^{-|I|-\sum_{i\in I}k_i},
 \quad
 \rho_I\prod_{j\notin I}\alpha_j\geq2^{-q-|\lambda|}.
\end{equation}

We first estimate $\widetilde c$. For all sufficiently large $|\lambda|$, we
prove by induction on $q-|I|$ that
$c_I\leq(12\cdot 2^{2q+|\lambda|}\ln|\lambda|)^{q-|I|}$ for every nonempty
$I\subseteq[q]$.
Indeed, the case $I=[q]$ follows from $c_{[q]}=1$. Suppose that
$\varnothing\ne I\subsetneq[q]$ and this estimate holds for every $J\supsetneq I$. Since
$q-|J|\leq q-|I|-1$ and there are at most $2^q$ such sets $J$,
\[
 \hat c_I=\sum_{J\supsetneq I}c_J
 \leq2^q
 \bigl(12\cdot 2^{2q+|\lambda|}\ln|\lambda|\bigr)^{q-|I|-1}.
\]
By Equality~\eqref{eq:c-recursion} and Inequality~\eqref{eq:alpha-rho-lower},
$c_I\leq(12\cdot 2^{q+|\lambda|}\ln|\lambda|)\hat c_I$. By combining this
inequality with the preceding estimate, we obtain the required bound for
$c_I$ and complete the induction. Since there are at most
$2^q$ nonempty sets $I\subseteq[q]$,
\begin{equation}\label{eq:ctilde-explicit-bound}
 \widetilde c
 =\sum_{\varnothing\ne I\subseteq[q]}c_I
 \leq2^q
 \bigl(12\cdot 2^{2q+|\lambda|}\ln|\lambda|\bigr)^{q-1}.
\end{equation}

We then estimate $\varepsilon_I(d)$. Recall that the terminal case in its
recursive definition is
$\varepsilon_{[q]}(d)=\exp(-\rho_{[q]}\sqrt d)$. By
Inequality~\eqref{eq:alpha-rho-lower},
\[
 \varepsilon_{[q]}(d)
 \leq\exp(-2^{-q-|\lambda|}\sqrt d)
 \leq\exp\left(-\frac{2^{-q-|\lambda|}\sqrt d}{16}\right).
\]
Now let $\varnothing\ne I\subsetneq[q]$. We have $\hat c_I\geq c_{[q]}=1$ and
$\prod_{j\notin I}\alpha_j\leq1$. By Equality~\eqref{eq:c-recursion},
$c_I\rho_I\geq12\ln|\lambda|\geq1$ for all sufficiently large
$|\lambda|$. Also,
$2^{-q-|\lambda|}\leq1$ and $\ln|\lambda|\geq1$. It follows that
\[
 \max\left\{
 \exp\left(-\frac{\hat c_I\sqrt d}{3}\right),
 \exp\left(-\frac{c_I\rho_I\sqrt d}{16}\right),\quad
 \exp(-\hat c_I\sqrt d\ln|\lambda|)
 \right\}
 \leq\exp\left(-\frac{2^{-q-|\lambda|}\sqrt d}{16}\right).
\]
By the bound
$\binom{k_j^2}{\lceil k_j^2/2\rceil}\leq2^{k_j^2}$ in the remaining term of
that definition, we have
\[
 \varepsilon_I(d)
 \leq3\exp\left(-\frac{2^{-q-|\lambda|}\sqrt d}{16}\right)
 +4\sum_{j\notin I}2^{k_j^2}\varepsilon_{I\cup\{j\}}(d).
\]

The first term on the right is the same for every $I$, and the terminal term
$\varepsilon_{[q]}(d)$ is at most this term. Each substitution adds one new
element to the index set, so after at most $q-|I|\leq q-1$ substitutions,
every resulting term, for some $J\subseteq[q]\setminus I$, is at most
\[
 3\exp\left(-\frac{2^{-q-|\lambda|}\sqrt d}{16}\right)
 \prod_{j\in J}\bigl(4\cdot2^{k_j^2}\bigr).
\]
The same set $J$ may occur more than
once if its elements are added in different orders. For every such $J$,
\[
 \prod_{j\in J}4\cdot2^{k_j^2}
 \leq4^q2^{\sum_jk_j^2}
 \leq2^{|\lambda|^2+2q}.
\]
At each substitution there are at most $q$ choices for the new element, so
the full expansion contains at most $1+q+\cdots+q^q$ terms. Consequently,
\begin{equation}\label{eq:error-asymptotic-bound}
 \varepsilon_I(d)
 \leq2^{|\lambda|^2+O_q(1)}
 \exp\left(-\frac{2^{-q-|\lambda|}\sqrt d}{16}\right)
\end{equation}
for every nonempty $I\subseteq[q]$ and all sufficiently large $|\lambda|$.

\medskip
\noindent
\emph{Proof of the second part of Theorem~\ref{thm:existence}.}
Let
\[
 D_\lambda=\left\lceil
 2^{2q|\lambda|+2|\lambda|^{2/3}}
 \right\rceil.
\]
We show that, for all sufficiently large $|\lambda|$,
conditions~\ref{cond:final-error} and~\ref{cond:final-union}
hold whenever $d\geq D_\lambda$.
If $d\geq D_\lambda$, then
$\sqrt d\geq2^{q|\lambda|+|\lambda|^{2/3}}$.
It follows from Inequality~\eqref{eq:ctilde-explicit-bound} that
\[
 \frac{\widetilde c}{\sqrt d}
 \leq
 \frac{2^q(12\cdot 2^{2q}\ln|\lambda|)^{q-1}}
 {2^{|\lambda|+|\lambda|^{2/3}}}
 =o(1)
\]
and
\[
 \frac{16\widetilde c\ln|\lambda|}{\sqrt d}
 \leq
 \frac{2^{q+4}(12\cdot 2^{2q})^{q-1}(\ln|\lambda|)^q}
 {2^{|\lambda|+|\lambda|^{2/3}}}
 \leq2^{-q-|\lambda|}
 \leq\prod_{i=1}^q\alpha_i
\]
for all sufficiently large $|\lambda|$. We have therefore verified
condition~\ref{cond:final-union}.

The number of half-sets of $X_i$ is at most $2^{k_i^2}$. Hence, by
Inequality~\eqref{eq:error-asymptotic-bound},
\[ 
 \sum_{i=1}^q
 \binom{k_i^2}{\lceil k_i^2/2\rceil}\varepsilon_{\{i\}}(d) \leq2^{2|\lambda|^2+O_q(1)}
 \exp\left(
 -\frac{2^{-q}2^{(q-1)|\lambda|+|\lambda|^{2/3}}}{16}
 \right) =o(1). 
\]
By combining this estimate with $\widetilde c/\sqrt d=o(1)$, we verify
condition~\ref{cond:final-error}.
Therefore the argument in the preceding subsection applies to every graph
$G$ with $\delta_q(G)\geq D_\lambda$. Thus
\[
 f(\lambda)\leq D_\lambda
 =2^{(2q+o(1))|\lambda|}.
\]
This completes the proof of Theorem~\ref{thm:existence}.

\section{\texorpdfstring{\thinspace}{}Proof of
Theorem~\ref{thm:equal-upper}}\label{sec:partition}

Throughout this section, $q$ is fixed. Let $G$ be a graph with
$\delta_q(G)\geq d>1$. We follow the container method of
Saxton and Thomason~\cite{ST2015}. We first construct an auxiliary weighted
graph from the $(q+1)$-cliques of $G$. We assign to each edge the number of
such cliques containing it. By using these weights, we can apply the minimum
$q$-codegree condition in the container argument. We then prove a weighted
container lemma and apply it to the color classes arising from the $q$ color
pools. We finish by choosing random lists and showing that no possible tuple
of containers can correspond to a coloring. We also use a lemma about
matrices, which we prove at the end of the section.

\subsection{The auxiliary weighted graph}\label{subsec:weighted-cliques}

Let $H$ be the spanning subgraph of $G$ whose edges are those contained in at
least one $(q+1)$-clique of $G$. For every $uv\in E(H)$, let $w(uv)$ be the
number of $(q+1)$-cliques containing $u$ and $v$, and let $(H,w)$ be the
resulting weighted graph. We next verify that $H$ is nonempty. Let $A$ consist
of any vertex of $G$. If $|A|<q$, then we choose a $q$-set $A'$ containing $A$. Since
$\delta_q(G)\geq d>1$, the set $N_G(A')$ is nonempty. Let $x\in N_G(A')$.
Then $x$ is adjacent to every vertex of $A$, so $A\cup\{x\}$ is a larger
clique. We replace $A$ by $A\cup\{x\}$ and repeat this step until $|A|=q$.
The set $N_G(A)$ is nonempty, and hence $A\cup\{y\}$ is a $(q+1)$-clique for
any $y\in N_G(A)$. Therefore $H$ is nonempty.

For $v\in V(H)$ and $S\subseteq V(H)$, let
\[
 d_{(H,w)}(v)=\sum_{u\in N_H(v)}w(uv),
 \quad
 e_w(S)=\sum_{uv\in E(H[S])}w(uv).
\]
In particular, $e_w(H)=e_w(V(H))$. When $e_w(H)>0$, let
\[
 \mu_w(S)=\frac{1}{2e_w(H)}\sum_{v\in S}d_{(H,w)}(v).
\]
For $v\in V(H)$, let $\mu_w(v)=\mu_w(\{v\})$. For a multiset $M$ of
$V(H)$, let $\mu_w(M)=\sum_{v\in M}\mu_w(v)$, where vertices are counted with
multiplicity. Let $\mathcal Q$ be the family of $(q+1)$-cliques in $G$, and
for $v\in V(G)$, let $\mathcal Q(v)=\{K\in\mathcal Q:v\in K\}$. Then
\begin{equation}\label{eq:clique-loads}
 d_{(H,w)}(v)=q|\mathcal Q(v)|,
 \quad
 e_w(H)=\binom{q+1}{2}|\mathcal Q|,
 \quad
 \mu_w(v)=\frac{|\mathcal Q(v)|}{(q+1)|\mathcal Q|}.
\end{equation}

We first claim that for every edge $uv$ of $H$, $d_{(H,w)}(v) \geq (d-1)w(uv)$. We prove this by double-counting the ordered pairs $(K,K')$ of distinct $(q+1)$-cliques
satisfying $\{u,v\}\subseteq K$ and $K \setminus\{u\}\subseteq K'$. For each
$K \in \mathcal{Q}$ containing $u$ and $v$, let $S = K\setminus\{u\}$. Since
$|N_G(S)|\geq d$ and $u\in N_G(S)$, the set
$K'= S\cup\{x\}$ is a $(q+1)$-clique distinct from $K$ for every
$x\in N_G(S)\setminus\{u\}$. Thus there are at least
$(d-1)w(uv)$ such pairs. 
On the other hand, for a fixed $K'$ containing $v$ but not $u$, there are at
most $q$ $q$-subsets of $K'$ containing $v$ that can form a $(q+1)$-clique
together with $u$. Thus the number of pairs is at most
$q|\mathcal Q(v)|=d_{(H,w)}(v)$. So the claim follows.
 
By interchanging $u$ and $v$, we also have
$d_{(H,w)}(u)\geq(d-1)w(uv)$. Consequently,
\begin{equation}\label{eq:clique-edge-weight}
 w(uv)\leq
 \frac{\min\{d_{(H,w)}(u),d_{(H,w)}(v)\}}{d-1}.
\end{equation}

Similarly, we can show that for every $v\in V(H)$,
\begin{equation}\label{eq:clique-vertex-measure}
 \mu_w(v)\leq\frac{1}{d-1}.
\end{equation}
Indeed, we count the ordered pairs $(K,K')$ of distinct $(q+1)$-cliques
satisfying $v\in K$ and $K\setminus\{v\}\subseteq K'$. Their number is at
least $(d-1)|\mathcal Q(v)|$. For each fixed $K'$, there are at most $q+1$
choices for $K$, so their number is at most $(q+1)|\mathcal Q|$. By
Equalities~\eqref{eq:clique-loads}, we obtain
Inequality~\eqref{eq:clique-vertex-measure}.

\subsection{A weighted container lemma and local measure estimates}

In this subsection, we prove the weighted container lemma used in the counting
argument and two auxiliary estimates needed later.
Let $(H,w)$ be the auxiliary weighted graph in Subsection~\ref{subsec:weighted-cliques},
and order its vertex set $V(H)=V(G)=\{v_1,v_2,\ldots,v_n\}$ so that
\[
 d_{(H,w)}(v_1)\geq d_{(H,w)}(v_2)\geq\cdots\geq d_{(H,w)}(v_n).
\]
For $j\in[n]$, let $V_j=\{v_1,v_2,\ldots,v_j\}$.

\begin{lemma}\label{lem:online-container}
For every $0<\zeta<1$, there is a function $C$ from the subsets of $V(H)$ to
the subsets of $V(H)$ such that, for every independent set $I\subseteq V(H)$,
there exists $T\subseteq I$ with
\begin{enumerate}[label=(C\arabic*),ref=(C\arabic*),labelindent=2em,
	leftmargin=*]
\item\label{cond:cover} $I\subseteq C(T)$.
\item\label{cond:edges} $e_w(C(T)) \leq 2\zeta e_w(H)$.
\item\label{cond:measure} $\mu_w(T) \leq 1/(\zeta(d-1))$.
\end{enumerate}
Moreover, $C(T)\cap V_j=C(T\cap V_j)\cap V_j$ for every
$T\subseteq V(H)$ and every $j\in[n]$.
For each independent set $I$, choose one such set $T$ and denote it by
$T(I)$. 

We call $T(I)$ the fingerprint of $I$ and $C(T(I))$ its container.
\end{lemma}

\begin{proof}
We use an algorithm similar to that of Saxton and Thomason~\cite{ST2015}. It has
two modes, called \emph{prune mode} and \emph{build mode}.
The set $A$ is the candidate part of the container. The auxiliary set $\Gamma$
consists of the vertices that have an earlier neighbor in $T$.

For the current auxiliary set $\Gamma$ and $i \in [n]$, let
\[
 F(v_i) = \{v_j \in V(H) : j > i,\ v_iv_j \in E(H),\ v_j \notin \Gamma\}.
\]
A vertex $v_i$ is \emph{active} if
\begin{equation}\label{eq:active-condition}
 \sum_{v_j \in F(v_i)} w(v_iv_j) \geq \zeta d_{(H,w)}(v_i).
\end{equation}
At each step, every set remains unchanged unless the algorithm explicitly
changes it.

\medskip
\noindent
\begin{minipage}{\textwidth}
\hrule
\smallskip
\noindent\textsc{Algorithm}
\par\smallskip
\noindent
\begin{minipage}[t]{0.48\textwidth}
\vspace{0pt}
\centering
Prune mode
\par\smallskip
\raggedright
\begin{enumerate}[leftmargin=1.5em,label=\arabic*.,itemsep=0.4em,topsep=0pt]
\item The input is an independent set $I$.
\item Initially, let $T = \varnothing$ and $\Gamma = \varnothing$.
\item For $i=1,2,\ldots,n$, calculate $F(v_i)$ using the current set $\Gamma$.
If $v_i \in I$ and $v_i$ is active, then add $v_i$ to $T$ and add all vertices
in $F(v_i)$ to $\Gamma$.
\item The output is $T(I) = T$.
\end{enumerate}
\end{minipage}
\hfill
\begin{minipage}[t]{0.48\textwidth}
\vspace{0pt}
\centering
Build mode
\par\smallskip
\raggedright
\begin{enumerate}[leftmargin=1.5em,label=\arabic*.,itemsep=0.4em,topsep=0pt]
\item The input is a set $T \subseteq V(H)$.
\item Initially, let $A = V(H)$ and $\Gamma = \varnothing$.
\item For $i=1,2,\ldots,n$, calculate $F(v_i)$ using the current set $\Gamma$.
If $v_i$ is active, then remove $v_i$ from $A$. If $v_i \in T$, then add all
vertices in $F(v_i)$ to $\Gamma$.
\item The output is $C(T) = (A \cup T) \setminus \Gamma$.
\end{enumerate}
\end{minipage}
\smallskip
\hrule
\end{minipage}
\medskip

Let $T=T(I)$ be the set produced by prune mode. We now show that prune mode
and build mode with input $T$ construct the same auxiliary set $\Gamma$ at
every step. We prove this by induction over the vertex order. Suppose that
the two auxiliary sets are the same before $v_i$ is examined. Then both modes
compute the same $F(v_i)$, so $v_i$ is active in both modes or neither. In
prune mode, $\Gamma$ is updated exactly when $v_i\in I$ and $v_i$ is active,
which is precisely the condition that puts $v_i$ into the final set $T(I)$.
In build mode, $\Gamma$ is updated exactly when $v_i\in T$. Hence the two
modes make the same update to $\Gamma$ at this step.

First, we prove~\ref{cond:cover}. Clearly $T(I)\subseteq I$, so it remains to
prove that $I\subseteq C(T(I))$. Let $v_i\in I$. We shall show that
$v_i\in C(T(I))$. Every vertex in $\Gamma$ has a neighbor of smaller index in
$T(I)$. Since $T(I)\subseteq I$ and $I$ is independent, no vertex of $T(I)$ is
adjacent to $v_i$. Hence $v_i$ is never added to $\Gamma$ in prune mode.
Since the two modes construct the same set $\Gamma$ at every step, $v_i$ is
never added to $\Gamma$ in build mode either. If $v_i\in T(I)$, then $v_i\in C(T(I))$ by definition. If
$v_i\notin T(I)$, then $v_i$ was inactive when it was examined in prune mode;
otherwise, it would have been added to $T(I)$. The build mode has the same
$\Gamma$ at that step, so $v_i$ is also inactive there and is not removed
from $A$. Thus $v_i\in A\setminus\Gamma\subseteq C(T(I))$, as required.
 
Next, we prove~\ref{cond:edges}. Let $A$ and $\Gamma$ be the final sets
produced by build mode with input $T$. Orient every edge of $H[C(T)]$ from $v_i$ to $v_j$
when $i < j$. Let $v_iv_j$ be an edge of $H[C(T)]$ with $i < j$. We claim that
$v_i \notin T$. Indeed, if $v_i \in T$, then $v_j \notin T$ because $T$ is
independent. Also $v_j \notin \Gamma$ because
$v_j \in C(T) = (A \cup T) \setminus \Gamma$. When build mode examined
$v_i$, we had
$v_j \in F(v_i)$, so $v_j$ would have been added to $\Gamma$. This is a
contradiction. Thus $v_i \notin T$. Since $v_i \in C(T)$, we have $v_i \in A$.

Since $v_i$ remains in the final set $A$, it was inactive when it was examined.
Therefore
\[
 \sum_{v_j \in F(v_i)} w(v_iv_j) < \zeta d_{(H,w)}(v_i).
\]
By counting each edge $v_iv_j\in E(H[C(T)])$ with $i<j$, we have
\[
 e_w(C(T))
 \leq \sum_{v_i \in A} \sum_{v_j \in F(v_i)} w(v_iv_j)
 < \zeta\sum_{v_i \in A} d_{(H,w)}(v_i)
 \leq 2\zeta e_w(H).
\]
Thus condition~\ref{cond:edges} holds.

Now we prove~\ref{cond:measure}. For distinct vertices $v_i,v_{j}\in T$,
the sets $F(v_i)$ and $F(v_{j})$ are disjoint, because each update only adds
vertices outside the current $\Gamma$. When $v_i$ entered $T$ in prune mode,
it was active, so by Inequality~\eqref{eq:active-condition} and
Inequality~\eqref{eq:clique-edge-weight}, we have
\[
 \sum_{v_j \in F(v_i)}d_{(H,w)}(v_j)
 \geq \sum_{v_j \in F(v_i)}(d-1)w(v_iv_j)
 = (d-1)\sum_{v_j \in F(v_i)} w(v_iv_j)
 \geq \zeta(d-1)d_{(H,w)}(v_i).
\]
By summing over $v_i \in T$ and using the disjointness of the sets $F(v_i)$,
we obtain
\[
 \zeta(d-1)\sum_{v_i \in T} d_{(H,w)}(v_i)
 \leq \sum_{v_i \in T} \sum_{v_j \in F(v_i)} d_{(H,w)}(v_j)
 \leq \sum_{v_i \in V(H)} d_{(H,w)}(v_i)
 = 2e_w(H).
\]
By dividing by $2e_w(H)$, condition~\ref{cond:measure} follows.

Finally, we prove the moreover part of the lemma. Let $T \subseteq V(H)$ and
$i \in [n]$, and
let $T' = T \cap V_i$. Apply build mode with inputs $T$ and $T'$, and let
$A,\Gamma$ and $A',\Gamma'$ denote the sets maintained in the two applications,
respectively. For every $j \leq i$, we have $v_j \in T$ if and only if
$v_j \in T'$. Initially, $A=A'=V(H)$ and $\Gamma=\Gamma'=\varnothing$.
Suppose that $A=A'$ and $\Gamma=\Gamma'$ before $v_j$ is examined. Then the
same set $F(v_j)$ is computed in both applications, so the update of $\Gamma$
and the decision on $v_j$ are also the same. By induction, $A=A'$ and
$\Gamma=\Gamma'$ immediately after $v_i$ is examined. At a later
step, suppose that $v_j$, where $j > i$, is examined. This step can only
remove $v_j$ from $A$ and add vertices $v_h$ with $h > j$ to
$\Gamma$. Hence
$A \cap V_i = A' \cap V_i$ and
$\Gamma \cap V_i = \Gamma' \cap V_i$. Therefore
\[
C(T) \cap V_i
 = \big((A \cap V_i) \cup (T \cap V_i)\big) \setminus (\Gamma \cap V_i)
 = \big((A' \cap V_i) \cup T'\big) \setminus (\Gamma' \cap V_i)
 = C(T') \cap V_i.
\]
Thus the moreover part holds, and we complete the proof of the lemma.
\end{proof}

For $\alpha>0$ and a multiset $S$ of $V(H)$, let
$D_\alpha(S)=\{v_j\in V(H):|S\cap V_j|\geq\alpha j\}$.
The next lemma bounds the weighted measure of $D_\alpha(S)$.

\begin{lemma}\label{lem:prefix}
Let $(H,w)$ be a weighted graph with $e_w(H) > 0$ and
$V(H) = \{v_1,v_2,\ldots,v_n\}$. Assume that
$d_{(H,w)}(v_1) \geq d_{(H,w)}(v_2) \geq \cdots
\geq d_{(H,w)}(v_n)$. Then, for every $\alpha>0$ and every multiset $S$ of
$V(H)$,
$\mu_w(D_\alpha(S)) \leq \mu_w(S)/\alpha$.
\end{lemma}

\begin{proof}
Let $\mu_i = \mu_w(v_i)$ for $i \in [n]$. Then
$\mu_1 \geq \mu_2 \geq \cdots \geq \mu_n \geq 0$ and
$\sum_i \mu_i = \mu_w(V(H)) = 1$.
Let $j \in [n]$ and $v_i \in D_\alpha(S) \cap V_j$. By the definition of
$D_\alpha(S)$, we have $|S \cap V_i| \geq \alpha i$. Since $i \leq j$, we
have $V_i \subseteq V_j$, and hence
$|S \cap V_i| \leq |S \cap V_j|$. Therefore
$i \leq |S \cap V_j|/\alpha$. Thus every vertex in
$D_\alpha(S) \cap V_j$ has index at most $|S \cap V_j|/\alpha$, and therefore
\begin{equation}\label{eq:prefix-count}
 |D_\alpha(S) \cap V_j| \leq \frac{|S \cap V_j|}{\alpha}.
\end{equation}

Let $\mu_{n+1} = 0$. Since
$\mu_i = \sum_{j = i}^n(\mu_j - \mu_{j+1})$, for any set
$A \subseteq V(H)$ we have
\begin{equation}\label{eq:prefix-identity}
 \mu_w(A) = \sum_{v_i \in A} \mu_i
 = \sum_{v_i \in A} \sum_{j = i}^n(\mu_j - \mu_{j+1})
 = \sum_{j = 1}^n(\mu_j - \mu_{j+1})|A \cap V_j|.
\end{equation}
For the last equality, we fix $j\in[n]$ and consider the term
$\mu_j-\mu_{j+1}$. It occurs once for each $v_i\in A$ with $i\leq j$, and
therefore occurs exactly $|A\cap V_j|$ times. The same argument applies when
$A$ is a multiset: each occurrence of $v_i$ contributes once to
$|A\cap V_j|$ for every $j\geq i$. Hence
Equality~\eqref{eq:prefix-identity} also holds for multisets, with
multiplicities. By Equality~\eqref{eq:prefix-identity} with
$A = D_\alpha(S)$, Inequality~\eqref{eq:prefix-count}, and
Equality~\eqref{eq:prefix-identity} with $A = S$, we have
\[
 \mu_w(D_\alpha(S))
 = \sum_{j = 1}^n(\mu_j - \mu_{j+1})|D_\alpha(S) \cap V_j|
 \leq \frac{1}{\alpha}\sum_{j = 1}^n(\mu_j - \mu_{j+1})|S \cap V_j|
 = \frac{\mu_w(S)}{\alpha}.
\]
This completes the proof.
\end{proof}

The next lemma shows that, for some $h$, a set of large weighted measure
contains a positive proportion of the first $h$ vertices. Moreover, the same
$h$ can be chosen so that the first $h$ vertices contain only a small total
number of elements from several sets of small weighted measure.

\begin{lemma}\label{lem:initial-segment}
Let $(H,w)$ be a weighted graph with $e_w(H) > 0$ and
$V(H) = \{v_1,v_2,\ldots,v_n\}$, where
$d_{(H,w)}(v_1) \geq d_{(H,w)}(v_2) \geq \cdots
\geq d_{(H,w)}(v_n)$. Let $s$ be a positive integer, let $r,\tau > 0$, and
let $0 < \beta \leq 1$. Suppose that $\mu_w(v_i)\leq 1/r$ for every $i\in[n]$. Let
$R,T_1,\ldots,T_s \subseteq V(H)$ satisfy $\mu_w(R) \geq \beta$ and
$\mu_w(T_j) \leq \tau$ for every $j \in [s]$. Let
$\kappa = \left\lfloor\beta r/(4(2 - \beta))\right\rfloor$. If
$\kappa \geq 1$,
then there is an index
$h > \kappa$ such that
\[
 |R \cap V_h| \geq \frac{\beta h}{2}
 \quad\text{and}\quad
 \sum_{j = 1}^s |T_j \cap V_h|
 < \frac{4s\tau(2 - \beta)}{\beta}h.
\]
\end{lemma}

\begin{proof}
Let $Z_1$ be the set of vertices $v_j$ for which
$|R \cap V_j| < \beta j/2$. If $v_j \in Z_1$, then
$|(V(H) \setminus R) \cap V_j| > (1 - \beta/2)j$, so
$Z_1 \subseteq D_{1 - \beta/2}(V(H) \setminus R)$. By
Lemma~\ref{lem:prefix}, we have
\[
 \mu_w(Z_1) \leq \frac{\mu_w(V(H) \setminus R)}{1 - \beta/2}
 \leq \frac{1 - \beta}{1 - \beta/2}
 = 1 - \frac{\beta}{2 - \beta}.
\]
Let $Z_2$ be the set of vertices $v_i$ for which
$\sum_j |T_j \cap V_i| \geq 4s\tau(2 - \beta)i/\beta$. The union of
$T_1,\ldots,T_s$ is treated as a multiset. By Lemma~\ref{lem:prefix}, we
have
\[
 \mu_w(Z_2)
 \leq \frac{\sum_j \mu_w(T_j)}{4s\tau(2 - \beta)/\beta}
 \leq \frac{\beta}{4(2 - \beta)}.
\]
Finally, since $\mu_w(v_i) \leq 1/r$ for every $i$, we have
$\mu_w(V_\kappa) \leq \kappa/r \leq \beta/(4(2 - \beta))$. Therefore
\[
 \mu_w(Z_1 \cup Z_2 \cup V_\kappa)
 \leq 1 - \frac{\beta}{2 - \beta}
 + \frac{\beta}{4(2 - \beta)} + \frac{\beta}{4(2 - \beta)}
 = 1 - \frac{\beta}{2(2 - \beta)} < 1.
\]
Some vertex $v_h$ lies outside this union. Then $h > \kappa$, and both
desired inequalities hold.
\end{proof}

\subsection{A random list assignment and completion of the proof}

We now prove the upper bound in Theorem~\ref{thm:equal-upper}. Fix
$0<\varepsilon<1/2$, and let $k$ be sufficiently large. Let
\begin{equation}\label{eq:d and ell-choice}
 d=\left\lceil\exp\bigl((1+2\varepsilon)k\ln(1/\rho_q)\bigr)\right\rceil, \quad 
 \ell = \left\lfloor
 (1 - \varepsilon)\frac{\ln(d - 1)}{\ln(1/\rho_q)}
 \right\rfloor.
\end{equation}

Since $\ln(d-1) \geq (1+2\varepsilon)k\ln(1/\rho_q)-\ln 2$ and $(1-\varepsilon)(1+2\varepsilon)
=1+\varepsilon(1-2\varepsilon)>1$ for all sufficiently large $k$, we have
\[
\ell\geq
\left\lfloor
(1-\varepsilon)(1+2\varepsilon)k
-\frac{(1-\varepsilon)\ln2}{\ln(1/\rho_q)}
\right\rfloor
 \geq k.
\]

Let $0 < \zeta < 1$ be a constant depending only on $q$ and $\varepsilon$, to be
fixed later. We apply Lemma~\ref{lem:online-container} with this value of $\zeta$. For
each independent set $I$ of $H$, let $T(I)$ be the set produced by the prune
algorithm and let $C(T(I))$ be the corresponding container. Then
\[
 I \subseteq C(T(I)),
 \quad
 e_w(C(T(I))) \leq 2\zeta e_w(H),
 \quad
 \mu_w(T(I)) \leq \frac{1}{\zeta(d - 1)},
\]
and the moreover part of Lemma~\ref{lem:online-container} holds for this
order. Let $\mathcal{F}$ be the
family of all sets $T(I)$ produced by the prune algorithm.

Let $t = \ell^2$. For $i \in [q]$, let
$X_i = \{(i,j) : j \in [t]\}$. For every vertex $u$ and every
$i \in [q]$, choose $L_i(u)$ independently and uniformly from
$\binom{X_i}{\ell}$, and let $L(u)=\bigcup_{i=1}^qL_i(u)$.
Then $L$ is a random $\equalparts{\ell}{q}$-list assignment with color pools
$X_1,\ldots,X_q$. We shall prove that $G$ has no $L$-coloring with positive
probability.

A tuple $\mathbf{T} = (T_{i,j})_{i \in [q],\,j \in [t]}$ is \emph{valid} if
$T_{i,j} \in \mathcal{F}$ for all $i,j$. Here $i$ denotes the color pool
$X_i$,
and $j$ is the index of the color $(i,j) \in X_i$. The associated
\emph{container tuple} is
$\mathbf{C} = (C_{i,j})_{i \in [q],\,j \in [t]}$, where
$C_{i,j} = C(T_{i,j})$. For $U\subseteq V(G)$, let
$\mathbf T\cap U=(T_{i,j}\cap U)_{i\in[q],\,j\in[t]}$, and define
$\mathbf C\cap U$ similarly. We say that $\mathbf{C}$ is \emph{rejected at $u$} if
$u \notin C_{i,j}$ for every $(i,j)\in L(u)$.

Suppose that $L$ has a proper coloring $\phi$. For $i \in [q]$ and
$j \in [t]$, let $I_{i,j} = \{u : \phi(u) = (i,j)\}$. Each $I_{i,j}$ is
independent in $G$. Since
$E(H) \subseteq E(G)$, each $I_{i,j}$ is also independent in $H$. Let
$T_{i,j} = T(I_{i,j})$ and $C_{i,j} = C(T_{i,j})$. Then $\mathbf{T}$ is valid and
$I_{i,j} \subseteq C_{i,j}$ for all $i,j$. For every vertex $u$, we have
$\phi(u) = (i,j)$ for some $(i,j)\in L(u)$. Hence
$u \in I_{i,j} \subseteq C_{i,j}$, so $\mathbf{C}$ is not rejected at $u$.
Thus every proper $L$-coloring produces a valid container tuple that is not
rejected at any vertex. Consequently, it is enough to prove that, with
positive probability, every valid container tuple is rejected at some
vertex.

The total number of valid tuples may be too large for a direct union bound.
For each valid tuple $\mathbf{T}$, we shall select an index $h(\mathbf T)$ so
that $V_{h(\mathbf T)}$ contains many vertices at which the associated
container tuple can be rejected, while the total size of the sets
$T_{i,j}\cap V_{h(\mathbf T)}$ is small. We then retain only the restriction
$\mathbf{T}\cap V_{h(\mathbf T)}$. By the moreover part of
Lemma~\ref{lem:online-container}, this restriction determines the container
tuple on $V_{h(\mathbf T)}$. We can therefore count the possible restrictions
instead of the full valid tuples.

\subsubsection{A large rejection set and local fingerprint bounds}

Let $\mathbf{T}$ be a valid tuple and let $\mathbf{C}$ be its associated
container tuple. We first find many vertices at which $\mathbf{C}$ can be
rejected with sufficiently large probability. For $i \in [q]$ and
$u \in V(G)$, let
\[
 a_i(u) = \frac{1}{t}|\{j \in [t] : u \in C_{i,j}\}|,
 \quad b_i(u) = 1 - a_i(u).
\]
Thus $ta_i(u)$ is the number of containers indexed by colors in $X_i$ that
contain $u$, while $tb_i(u)$ is the number of these containers that avoid $u$.
For $i\in[q]$ and each $(q+1)$-clique $K\in\mathcal Q$, let
\[
 \xi_i(K) = \max\left\{\sum_{u \in K} a_i(u) - 1,0\right\}.
\]
With the convention that $\binom{z}{2}=0$ when $z<2$, we have
$\max\{z-1,0\}\leq\binom{z}{2}$ for every nonnegative integer $z$. Therefore,
\begin{equation}\label{eq:xi-bound}
 \xi_i(K)
 = \max\left\{\frac{1}{t}\sum_{j = 1}^t |K \cap C_{i,j}| - 1,0\right\}
 \leq \frac{1}{t}\sum_{j = 1}^t \binom{|K \cap C_{i,j}|}{2}.
\end{equation}
For fixed $i,j$, consider the pairs $(K,uv)$ such that
$K \in \mathcal Q$, $uv \in E(H[C_{i,j}])$ and $\{u,v\} \subseteq K$.
For each $K$, there are $\binom{|K \cap C_{i,j}|}{2}$ choices for $uv$.
For each $uv \in E(H[C_{i,j}])$, there are $w(uv)$ choices for $K$. Hence
\[
\sum_{K\in\mathcal Q}\binom{|K\cap C_{i,j}|}{2}
=\sum_{uv\in E(H[C_{i,j}])}w(uv)
=e_w(C_{i,j}).
\] 
By Inequality~\eqref{eq:xi-bound}, the above identity,
condition~\ref{cond:edges} of Lemma~\ref{lem:online-container} and
Equalities~\eqref{eq:clique-loads}, we have
\begin{equation}\label{eq:excess}
 \begin{aligned}
 \sum_{K \in \mathcal Q} \xi_i(K)
 & \leq \frac{1}{t}\sum_{K \in \mathcal Q} \sum_{j = 1}^t
       \binom{|K \cap C_{i,j}|}{2}\\
 &= \frac{1}{t}\sum_{j = 1}^t \sum_{K \in \mathcal Q}
       \binom{|K \cap C_{i,j}|}{2}\\
 &= \frac{1}{t}\sum_{j = 1}^t e_w(C_{i,j})
 \leq 2\zeta e_w(H)
 = 2\zeta\binom{q + 1}{2}|\mathcal Q|.
 \end{aligned}
\end{equation}
 
Let $\gamma = (2q\binom{q + 1}{2}\zeta)^{1/2}$, where $\zeta$ is
sufficiently small that $\gamma < 1$. By summing Inequality~\eqref{eq:excess} over $i\in[q]$, we have
\[
\sum_{K\in\mathcal Q}\sum_{i=1}^q\xi_i(K)
=\sum_{i=1}^q\sum_{K\in\mathcal Q}\xi_i(K)
\leq 2q\zeta\binom{q+1}{2}|\mathcal Q|
=\gamma^2|\mathcal Q|.
\]
Thus at most $\gamma|\mathcal Q|$ cliques $K\in\mathcal Q$ satisfy
$\sum_{i=1}^q\xi_i(K)>\gamma$, since otherwise their total contribution
would exceed $\gamma^2|\mathcal Q|$. Therefore, at least
$(1-\gamma)|\mathcal Q|$ cliques $K\in\mathcal Q$ satisfy
$\sum_{i=1}^q\xi_i(K)\leq\gamma$.
  
To find a vertex at which the container tuple is likely to be rejected, we
use the following lemma, whose proof is postponed until the end of the
section. Recall that $\rho_q$ is the unique $x\in(0,1)$ satisfying
$x=(1-x)^q$.

\begin{lemma}\label{lem:robust-array}
Let $X=(x_{ij})_{0\leq i\leq q,\,j\in[q]}$ be a $(q+1)\times q$ array with
entries in $[0,1]$, and let
$\xi_1,\ldots,\xi_q \geq 0$. Suppose that
$\sum_{i = 0}^q x_{ij} \leq 1 + \xi_j$ for every $j$. Then there exists some 
$i \in \{0,1,\ldots,q\}$ such that
\[
 \prod_{j = 1}^q(1 - x_{ij}) \geq \rho_q - \sum_{j = 1}^q\xi_j.
\]
\end{lemma} 

Let $K = \{u_0,u_1,\ldots,u_q\} \in \mathcal Q$ be such that
$\sum_{j = 1}^q \xi_j(K) \leq \gamma$. For each $j \in [q]$, we have
$\sum_{u \in K} a_j(u) \leq 1 + \xi_j(K)$. By
Lemma~\ref{lem:robust-array}, applied to the array
$(a_j(u_i))_{0 \leq i \leq q,\,j \in [q]}$, there is
$i_0 \in \{0,1,\ldots,q\}$ such that
\[
 \prod_{j = 1}^q b_j(u_{i_0})
 = \prod_{j = 1}^q(1 - a_j(u_{i_0}))
 \geq \rho_q - \sum_{j = 1}^q \xi_j(K)
 \geq \rho_q - \gamma.
\]
Let $R(\mathbf C)$ be the following set of vertices:
\[
 R(\mathbf{C})
 = \left\{u \in V(G) : \prod_{i = 1}^q b_i(u)
 \geq \rho_q - \gamma\right\}.
\]
At least $(1 - \gamma)|\mathcal Q|$ members of $\mathcal Q$ therefore
contain a vertex of $R(\mathbf{C})$. By double-counting the pairs $(u,K)$ with
$u\in R(\mathbf C)\cap K$, we have
\[
\sum_{u\in R(\mathbf C)}|\mathcal Q(u)|
=\sum_{K\in\mathcal Q}|K\cap R(\mathbf C)|
\geq(1-\gamma)|\mathcal Q|.
\]
Let $\beta = (1 - \gamma)/(q + 1)$.
By Equalities~\eqref{eq:clique-loads}, we have
\[
 \mu_w(R(\mathbf{C}))
 = \frac{\sum_{u \in R(\mathbf{C})}|\mathcal Q(u)|}
 {(q + 1)|\mathcal Q|}
 \geq \frac{1 - \gamma}{q + 1} = \beta.
\]

We now choose the appropriate set $V_h$ and at the same time control the
restriction of $\mathbf{T}$ to this set.
Recall that $\mathbf{T}$ consists of the $qt$ sets $T_{i,j}$ and that
$\mu_w(T_{i,j}) \leq 1/(\zeta(d - 1))$ for all $i \in [q]$ and $j \in [t]$. By
Inequality~\eqref{eq:clique-vertex-measure},
$\mu_w(v_i) \leq 1/(d - 1)$ for every $i\in[n]$. Let
\[
 \kappa = \left\lfloor\frac{\beta(d - 1)}{4(2 - \beta)}\right\rfloor,
 \quad
 \theta = \frac{4qt(2 - \beta)}{\beta\zeta(d - 1)}.
\]
For sufficiently large $d$, we have $\kappa \geq 1$. We apply
Lemma~\ref{lem:initial-segment} with $r = d - 1$,
$\tau = 1/(\zeta(d - 1))$ and $s = qt$. The set $R$ in that lemma is
$R(\mathbf C)$, and $T_1,\ldots,T_s$ are the $qt$ sets $T_{i,j}$, listed
separately in any order. There is an index
$h > \kappa$ such that
\begin{equation}\label{eq:local-restrictions}
 |R(\mathbf{C}) \cap V_h| \geq \frac{\beta h}{2},
 \quad
 \sum_{i = 1}^q \sum_{j = 1}^t |T_{i,j} \cap V_h| < \theta h.
\end{equation}
There may be more than one index satisfying~\eqref{eq:local-restrictions}.
For each valid tuple $\mathbf{T}$, let $h(\mathbf{T})$ be the smallest such
index.

For each possible value of $h(\mathbf T)$, consider the valid tuples having
this value. We now count their possible restrictions
$\mathbf T\cap V_{h(\mathbf T)}$. Each such restriction corresponds to the
set of triples $(i,j,u)\in[q]\times[t]\times V_{h(\mathbf T)}$ for which
$u\in T_{i,j}\cap V_{h(\mathbf T)}$. By the second inequality
in~\eqref{eq:local-restrictions}, we have
\[
 \left|\{(i, j, u) : i \in [q],\ j \in [t],
 u \in T_{i,j} \cap V_{h(\mathbf T)}\}\right|
 = \sum_{i = 1}^q \sum_{j = 1}^t
 |T_{i,j} \cap V_{h(\mathbf T)}| < \theta h(\mathbf T).
\]
There are $qth(\mathbf T)$ triples in
$[q]\times[t]\times V_{h(\mathbf T)}$. Since $h(\mathbf T)>\kappa$, for all
sufficiently large $d$ we have
$1\leq\theta h(\mathbf T)\leq qth(\mathbf T)/2$. By the standard binomial
sum estimate, the number of possible restrictions is at most
\begin{equation}\label{eq:tuple-count}
 \begin{aligned}
 \sum_{z < \theta h(\mathbf T)} \binom{q t h(\mathbf T)}{z}
 & \leq \left(\frac{\ee q t h(\mathbf T)}
 {\theta h(\mathbf T)}\right)^{\theta h(\mathbf T)}
 = \exp\left\{\theta h(\mathbf T)\ln \frac{\ee q t}{\theta}\right\}\\
 &= \exp\left\{\frac{4q(2 - \beta)}{\beta\zeta}
 \frac{h(\mathbf T)t}{d - 1}
 \ln \frac{\ee\beta\zeta(d - 1)}{4(2 - \beta)}\right\}\\
 &= \exp\left\{\frac{4q(2 - \beta)}{\beta\zeta}
 \frac{h(\mathbf T)t}{d - 1}
 \left(\ln(d - 1) + \ln \frac{\ee\beta\zeta}{4(2 - \beta)}\right)\right\}\\
 & \leq \exp\left\{c_{q,\zeta}
 \frac{h(\mathbf T)t\ln(d - 1)}{d - 1}\right\}.
 \end{aligned}
\end{equation}
Here $c_{q,\zeta} = 8q(2 - \beta)/(\beta\zeta)$ depends only on $q$ and
$\zeta$.

\subsubsection{Rejection probability and the union bound}

For a valid tuple $\mathbf T$ and $h\in[n]$, let
$\mathcal A(\mathbf T,h)$ be the event that $\mathbf C\cap V_h$ is not
rejected at any vertex of $R(\mathbf C)\cap V_h$, where $\mathbf C$ is the
container tuple determined by $\mathbf T$. By the moreover part of
Lemma~\ref{lem:online-container}, for every $i\in[q]$ and $j\in[t]$, we have
$C_{i,j}\cap V_h=C(T_{i,j}\cap V_h)\cap V_h$. Therefore
$\mathbf T\cap V_h$ determines $\mathbf C\cap V_h$,
$R(\mathbf C)\cap V_h$, and the event $\mathcal A(\mathbf T,h)$.

Since a proper $L$-coloring produces a valid container tuple that is not
rejected at any vertex, $L$ can be colorable only if
$\mathcal A(\mathbf T,h(\mathbf T))$ occurs for some valid tuple $\mathbf T$.
We shall show that the probability of this is less than one.

We first estimate the probability of each event. Let $\delta = (\ell - 1)/t$.
We begin with a binomial ratio estimate. For any $y$ such that
$ty \in \{0,1,\ldots,t\}$, we have
\begin{equation}\label{eq:binomial-ratio-bound}
 \frac{\binom{ty}{\ell}}{\binom{t}{\ell}}
 \geq \max\{y - \delta,0\}^{\ell}.
\end{equation}
Indeed, if $ty < \ell$, then $ty \leq \ell - 1$ and hence $y \leq \delta$, so both
sides are zero. If $ty \geq \ell$, then $ty-i\geq ty-\ell+1$ and
$t-i\leq t$ for every $0\leq i\leq\ell-1$. Therefore
\[
 \frac{\binom{ty}{\ell}}{\binom{t}{\ell}}
 = \frac{ty(ty-1)\cdots(ty-\ell+1)}
 {t(t-1)\cdots(t-\ell+1)}
 = \prod_{i = 0}^{\ell - 1}\frac{ty - i}{t - i}
 \geq \left(\frac{ty - \ell + 1}{t}\right)^{\ell}
 = (y - \delta)^{\ell}.
\]

Let $\mathbf T$ be a valid tuple. For any vertex
$u\in R(\mathbf{C})\cap V_{h(\mathbf T)}$ and each $i \in [q]$, there are
exactly $tb_i(u)$
colors $(i,j)\in X_i$ for which $u \notin C_{i,j}$. Since $L_i(u)$ is chosen
uniformly from $\binom{X_i}{\ell}$, the probability that
$u \notin C_{i,j}$ for every $(i,j)\in L_i(u)$ is
$\binom{tb_i(u)}{\ell}/\binom{t}{\ell}$. By
Inequality~\eqref{eq:binomial-ratio-bound}, this probability is at least
$\max\{b_i(u) - \delta,0\}^{\ell}$.

To combine these estimates over the $q$ color pools, we use the following
elementary product inequality.

\begin{lemma}\label{lem:product-estimate}
Let $0 \leq x_i \leq 1$ and $y_i \geq 0$ for $i \in [q]$. Then
\[
 \prod_{i = 1}^q \max\{x_i - y_i,0\}
 \geq
 \prod_{i = 1}^q x_i - \sum_{i = 1}^q y_i.
\]
\end{lemma}

\begin{proof}
Let $z_i=\max\{x_i-y_i,0\}$. Then $0\leq z_i\leq x_i\leq1$ and
$x_i-z_i\leq y_i$ for every $i\in[q]$. The difference between the two
products can be written as
\[
 \prod_{i=1}^q x_i-\prod_{i=1}^q z_i
 =\sum_{i=1}^q \left[(x_i-z_i)
   \left(\prod_{j<i}z_j \right) \left(\prod_{j>i}x_j \right) \right]
 \leq\sum_{i=1}^q y_i.
\]
Because all the factors in the two products are at most one, the desired
inequality follows.
\end{proof}

By applying Lemma~\ref{lem:product-estimate} with $x_i=b_i(u)$ and
$y_i=\delta$, and by the definition of $R(\mathbf{C})$, we have
\[
 \prod_{i = 1}^q \max\{b_i(u) - \delta,0\}
 \geq \prod_{i = 1}^q b_i(u) - q\delta
 \geq \rho_q - \gamma - q\delta.
\]
Since the lists in different pools are independent, the probability that
$\mathbf{C}$ is rejected at $u$ is at least
\[
 \left[\prod_{i = 1}^q
 \max\{b_i(u) - \delta,0\}\right]^{\ell}
 \geq (\rho_q - \gamma - q\delta)^{\ell}.
\]

Since $\gamma\to0$ as $\zeta\to0$, we choose $\zeta$ sufficiently small so that $\rho_q - 2\gamma > 0$ and
\begin{equation}\label{eq:zeta-choice}
 (1 - \varepsilon)
 \frac{\ln(1/(\rho_q - 2\gamma))}{\ln(1/\rho_q)}
 < 1 - \frac{\varepsilon}{2}.
\end{equation}
By Equality~\eqref{eq:d and ell-choice}, $\ell \to \infty$ as $d \to \infty$. Since
$t = \ell^2$, we have
$q\delta = q(\ell - 1)/t \leq q/\ell = o(1)$. Thus $q\delta \leq \gamma$ for all
sufficiently large $d$. Let $p=(d-1)^{-1+\varepsilon/2}$.
By Equality~\eqref{eq:d and ell-choice} and
Inequality~\eqref{eq:zeta-choice}, and since
$0 < \rho_q - 2\gamma < 1$, we have
\begin{equation}\label{eq:reject-power}
 \begin{aligned}
 (\rho_q - \gamma - q\delta)^{\ell}
 & \geq (\rho_q - 2\gamma)^{\ell}\\
 & \geq (\rho_q - 2\gamma)^{
 (1 - \varepsilon)\ln(d - 1)/\ln(1/\rho_q)}\\
 &= \exp\left\{-(1 - \varepsilon)
 \frac{\ln(1/(\rho_q - 2\gamma))}{\ln(1/\rho_q)}
 \ln(d - 1)\right\}\\
 & \geq \exp\left\{\left(-1 + \frac{\varepsilon}{2}\right)
 \ln(d - 1)\right\}\\
 &= (d-1)^{-1+\varepsilon/2} = p.
 \end{aligned}
\end{equation}
By the first inequality in~\eqref{eq:local-restrictions}, the set
$R(\mathbf{C})\cap V_{h(\mathbf T)}$ has size at least
$\beta h(\mathbf T)/2$. The random lists at
distinct vertices are independent. By Inequality~\eqref{eq:reject-power},
$\mathbf{C}$ is rejected at each vertex of this set with probability at least
$p$. Hence
\begin{equation}\label{eq:tuple-survive}
 \Prob\bigl(\mathcal A(\mathbf T,h(\mathbf T))\bigr)
 \leq (1 - p)^{\beta h(\mathbf T)/2}
 \leq \exp\{-\beta h(\mathbf T)p/2\}.
\end{equation}

Among the valid tuples with the same value of $h(\mathbf T)$,
Inequality~\eqref{eq:tuple-count} bounds the number of distinct events
$\mathcal A(\mathbf T,h(\mathbf T))$, while
Inequality~\eqref{eq:tuple-survive} bounds the probability of each such
event. Before applying the union bound, we show that the exponent in
Inequality~\eqref{eq:tuple-count} is at most
$\beta h(\mathbf T)p/4$.

Recall that $\zeta$ is a fixed positive constant depending only on $q$ and
$\varepsilon$, and that
\[
 \gamma=\left(2q\binom{q+1}{2}\zeta\right)^{1/2},
 \quad
 \beta=\frac{1-\gamma}{q+1}.
\]
Hence $\beta$ and $c_{q,\zeta}$ are also fixed positive constants depending
only on $q$ and $\varepsilon$. By the definitions of $t$ and $p$, and the
choice of $\ell$ in~\eqref{eq:d and ell-choice}, we have
\[
 \frac{c_{q,\zeta}t\ln(d-1)/(d-1)}{\beta p/4}
 \leq
 \frac{4c_{q,\zeta}(1-\varepsilon)^2}
 {\beta\ln^2(1/\rho_q)}
 \frac{\ln^3(d-1)}{(d-1)^{\varepsilon/2}}
 \longrightarrow 0
 \quad\text{as }d\to\infty.
\]
Thus, for all sufficiently large $d$ and every valid tuple $\mathbf T$,
\begin{equation}\label{eq:count-exponent-comparison}
 c_{q,\zeta}\frac{h(\mathbf T)t\ln(d-1)}{d-1}
 \leq\frac{\beta h(\mathbf T)p}{4}.
\end{equation}

We now group the valid tuples according to the value of $h(\mathbf T)$.
Since $\kappa<h(\mathbf T)\leq n$, the union bound,
Inequalities~\eqref{eq:tuple-count} and~\eqref{eq:tuple-survive}, and
Inequality~\eqref{eq:count-exponent-comparison} imply that
\[
 \begin{aligned}
 \Prob\left(\bigcup_{\mathbf T\text{ valid}}
 \mathcal A(\mathbf T,h(\mathbf T))\right)
 &\leq \sum_{\kappa<h\leq n}
 \exp\left\{c_{q,\zeta}\frac{ht\ln(d-1)}{d-1}\right\}
 \exp\{-\beta hp/2\}\\
 &\leq \sum_{\kappa<h\leq n}\exp\{-\beta hp/4\}
 \leq \sum_{h>\kappa}\exp\{-\beta hp/4\}.
 \end{aligned}
\]
This is the tail of a geometric series. For all sufficiently large $d$, we
have $1 - \exp\{-\beta p/4\} \geq \beta p/8$ and
$\kappa \geq \beta(d - 1)/(8(2 - \beta))$. Therefore,
\[
 \begin{aligned}
 \sum_{h > \kappa}\exp\{-\beta h p/4\}
 &= \frac{\exp\{-\beta(\kappa + 1)p/4\}}
 {1 - \exp\{-\beta p/4\}}
 \leq \frac{8}{\beta p}\exp\{-\beta\kappa p/4\}\\
 & \leq \frac{8}{\beta}(d - 1)^{1 - \varepsilon/2}
 \exp\left\{-\frac{\beta^2}{32(2 - \beta)}
 (d - 1)^{\varepsilon/2}\right\} \longrightarrow 0
 \quad\text{as }d\to\infty.
 \end{aligned}
\]

Thus, with positive probability, none of the events
$\mathcal A(\mathbf T,h(\mathbf T))$ occurs. For such a list assignment $L$,
every valid container tuple is rejected at a vertex of
$R(\mathbf C)\cap V_{h(\mathbf T)}$. Since a proper $L$-coloring would produce
a valid container tuple that is not rejected at any vertex, $L$ is
uncolorable.
Since $\ell\geq k$, the graph $G$ is not $\equalparts{k}{q}$-choosable.
Since $G$ was arbitrary,
\[
 f(\equalparts{k}{q})
 \leq\left\lceil\exp\bigl((1+2\varepsilon)k\ln(1/\rho_q)\bigr)\right\rceil
\]
for every sufficiently large $k$.
Since $\varepsilon$ is arbitrary,
\[
 f(\equalparts{k}{q})\leq\rho_q^{-(1+o(1))k}.
\]
We have now proved Theorem~\ref{thm:equal-upper} apart from
Lemma~\ref{lem:robust-array}, which we prove next.

\subsection{Proof of Lemma~\ref{lem:robust-array}}\label{subsec:array-proof}

Now we prove Lemma~\ref{lem:robust-array}. We need the following elementary
lemma for arrays whose column sums are one.

\begin{lemma}\label{lem:exact-array}
Let $P = (p_{ij})$ be a $(q + 1)\times q$ array with nonnegative entries and
$\sum_{i = 0}^q p_{ij} = 1$ for every $j \in [q]$. Then some row
$i \in \{0,1,\ldots,q\}$ satisfies
\begin{equation}\label{eq:allocation}
        \prod_{j = 1}^q(1 - p_{ij}) \geq \rho_q.
\end{equation}
The constant $\rho_q$ is best possible, and equality is attained by the array
with $p_{0j} = \rho_q$, $p_{jj} = 1 - \rho_q$, and $p_{ij} = 0$ whenever
$i \notin \{0,j\}$.
\end{lemma}

\begin{proof}
For $q = 1$, we have $\rho_1 = 1/2$. Since $p_{01} + p_{11} = 1$, one of
$p_{01}$ and $p_{11}$ is at most $1/2$. Hence one of $1 - p_{01}$ and
$1 - p_{11}$ is at least $1/2$. Equality is attained when
$p_{01} = p_{11} = 1/2$.

Assume now that $q \geq 2$. Let $F(x)=(1-x)^q-x$ and
$a=1-\rho_q$. Then $\rho_q$ is the unique zero of $F$ in $(0,1)$. Since
$F(1/2) = 2^{-q} - 1/2 < 0$ and $F(0) > 0$, the root satisfies $\rho_q < 1/2$, and
so $a > 1/2$. Also we have $a^q = (1 - \rho_q)^q = \rho_q$. Let
\[
 g(x) =
 \begin{cases}
  \ln(1 - x)/\ln\rho_q, & 0 \leq x \leq a,\\
  1, & a \leq x \leq 1.
 \end{cases}
\]
The two definitions agree at $x = a$, since $1 - a = \rho_q$. On $[0,a]$, the
function $g$ is increasing and convex, $g(0) = 0$, and
$g(\rho_q) = \ln(1 - \rho_q)/\ln\rho_q = \ln a/\ln a^q = 1/q$.

We first claim that for any $x_0,\ldots,x_q \in [0,1]$ with
$\sum_{i = 0}^q x_i = 1$, the following holds.
\begin{equation}\label{eq:column-bound}
        \sum_{i = 0}^q g(x_i) \leq 1 + \frac{1}{q}.
\end{equation}
We consider the following two cases.

\medskip
\noindent
\textbf{Case 1.} There is $i \in \{0,\ldots,q\}$ such that $x_i \geq a$.
\medskip

We may relabel the coordinates and assume that $x_0 \geq a$. Let
$m = 1 - x_0$. Then
$g(x_0) = 1$, $0 \leq m \leq 1 - a = \rho_q < a$, and
$\sum_{i = 1}^q x_i = m$. If $m = 0$,
then $\sum_{i = 0}^q g(x_i) = 1$. Assume $m > 0$. For
$1 \leq i \leq q$,
\[
        x_i = \frac{x_i}{m}m + \left(1 - \frac{x_i}{m}\right)0.
\]
By the convexity of $g$ on $[0,a]$ and by $g(0) = 0$, we have
\[
        g(x_i) \leq
        \frac{x_i}{m}g(m) + \left(1 - \frac{x_i}{m}\right)g(0)
 = \frac{x_i}{m}g(m).
\]
Thus
\[
        \sum_{i = 0}^q g(x_i)
 = 1 + \sum_{i = 1}^q g(x_i)
 \leq 1 + g(m)
 \leq 1 + g(\rho_q)
 = 1 + \frac{1}{q}.
\]

\medskip
\noindent
\textbf{Case 2.} $x_i < a$ for every $i \in \{0,\ldots,q\}$.
\medskip

Let $h(z_0,\ldots,z_q) = \sum_{i = 0}^q g(z_i)$, and consider the set
\[
K = \left\{(z_0,\ldots,z_q) : 0 \leq z_i \leq a,
      \sum_{i = 0}^q z_i = 1\right\}.
\]
Since $g$ is convex on $[0,a]$,
the function $h$ is convex on $K$. Hence the maximum of $h$ on $K$ is attained
at an extreme point of $K$.

We first describe the extreme points of $K$.

Let $\mathbf{w} = (w_0,\ldots,w_q) \in K$ be an extreme point. First, $\mathbf{w}$
has at most one coordinate in $(0,a)$. Indeed, by relabeling, assume that
$w_0,w_1 \in (0,a)$. Choose $\epsilon > 0$ small enough that both
\[
\mathbf{w}' = (w_0 - \epsilon,w_1 + \epsilon,w_2,\ldots,w_q)
\quad\text{and}\quad
\mathbf{w}'' = (w_0 + \epsilon,w_1 - \epsilon,w_2,\ldots,w_q)
\]
belong to $K$. Moreover,
\[
\mathbf{w} = \frac{1}{2}\mathbf{w}' + \frac{1}{2}\mathbf{w}''.
\]
This is impossible for an extreme point of $K$. Hence $\mathbf{w}$ has at most
one coordinate in $(0,a)$.

Since $a > 1/2$ and the coordinates of $\mathbf{w}$ sum to $1$, at most one
coordinate can be equal to $a$. We next show that some coordinate must be
equal to $a$. Assume this is not the case. Since every coordinate lies in
$[0,a]$ and none is equal to $a$, every coordinate is either $0$ or lies in
$(0,a)$. Since there is at most one coordinate in $(0,a)$, at most one
coordinate is positive. But the coordinates sum to $1$, so this positive
coordinate would have to be $1$. This is impossible, because every coordinate
is at most $a < 1$.

Thus exactly one coordinate of $\mathbf{w}$ is equal to $a$. We may relabel
the coordinates and assume that $w_0 = a$. Then
\[
\sum_{i = 1}^q w_i = 1 - a = \rho_q.
\]
No coordinate among $w_1,\ldots,w_q$ can be equal to $a$, since $w_0 = a$
and at most one coordinate can equal $a$. Also, at most one of
$w_1,\ldots,w_q$ can be positive. Otherwise two of them would lie in
$(0,a)$. Hence exactly one of $w_1,\ldots,w_q$ is positive. Its value must
be $1 - a$, and all the other coordinates are $0$. Therefore
$\mathbf{w}$ is a permutation of $(a,1 - a,0,\ldots,0)$.

Thus every extreme point of $K$ has this form. Since the maximum of $h$ on
$K$ is attained at an extreme point, we have
\[
 \sum_{i = 0}^q g(x_i)
 \leq g(a) + g(1 - a)
 = g(a) + g(\rho_q)
 = 1 + \frac{1}{q}.
\]
We have proved Inequality~\eqref{eq:column-bound}.

\medskip

Now suppose, for a contradiction, that
$\prod_{j = 1}^q(1 - p_{ij}) < \rho_q$ for every $i \in \{0,\ldots,q\}$. Let
$S_i = \sum_{j = 1}^q g(p_{ij})$. If row $i$ contains an entry at least $a$, then
$S_i \geq 1$. If it contains no entry at least $a$, then all its entries lie
in $[0,a]$. Since $\ln\rho_q < 0$ and
$\prod_{j = 1}^q(1 - p_{ij}) < \rho_q$, we have
\[
        S_i = \frac{\ln\prod_{j = 1}^q(1 - p_{ij})}{\ln\rho_q} > 1.
\]
Also, $a > 1/2$, so each column contains at most one entry at least $a$. There are only
$q$ columns but $q + 1$ rows, so at least one row contains no entry at
least $a$. Hence at least one $S_i$ is strictly larger than one, while all
the others are at least one. Therefore $\sum_{i = 0}^q S_i > q + 1$.

On the other hand, by Inequality~\eqref{eq:column-bound} for each column of $P$,
we have
\[
        \sum_{i = 0}^q S_i =
        \sum_{j = 1}^q \sum_{i = 0}^q g(p_{ij})
 \leq q\left(1 + \frac{1}{q}\right) = q + 1,
\]
a contradiction. Thus some row satisfies
Inequality~\eqref{eq:allocation}.

For the array stated in the lemma, the product is
$(1 - \rho_q)^q = \rho_q$ in row $0$ and $\rho_q$ in each row
$i \in [q]$. Hence the
constant cannot be improved.
\end{proof}

\begin{proof}[Proof of Lemma~\ref{lem:robust-array}]
The column sums of the array $X = (x_{ij})$ may be slightly larger than one.
We first rescale the columns so that their sums are at
most one, and then add nonnegative entries to make each column sum to one. For each $i \in \{0,1,\ldots,q\}$ and $j \in [q]$, set
\[
p_{ij}=\frac{x_{ij}}{1+\xi_j}.
\]
For each $j\in[q]$, let
\[
s_j=\sum_{i=0}^q p_{ij},
\quad
r_{ij}=
\begin{cases}
	1-s_j, & i=0,\\
	0, & 1\leq i\leq q,
\end{cases}
\]
and set $\widehat p_{ij}=p_{ij}+r_{ij}$ for every
$i\in\{0,1,\ldots,q\}$ and $j\in[q]$.
Then $0 \leq p_{ij} \leq 1$ and $s_j \leq 1$, and
$\widehat P = (\widehat p_{ij})$ has nonnegative entries, every column of
$\widehat P$ sums to one, and $\widehat p_{ij} \geq p_{ij}$ for all $i, j$. By
Lemma~\ref{lem:exact-array}, there is a row $i$ such that
\begin{equation}\label{eq:p-product-lower}
        \prod_{j = 1}^q(1 - p_{ij})
 \geq
        \prod_{j = 1}^q(1 - \widehat p_{ij})
 \geq \rho_q.
\end{equation}
Since
$x_{ij} = (1 + \xi_j)p_{ij}$, we have
$1 - x_{ij} = 1 - p_{ij} - \xi_j p_{ij}
\geq 1 - p_{ij} - \xi_j$. Also $1 - x_{ij} \geq 0$. Hence
$1 - x_{ij} \geq \max\{1 - p_{ij} - \xi_j,0\}$.
By Lemma~\ref{lem:product-estimate}, with $x_j = 1 - p_{ij}$ and
$y_j = \xi_j$, and by Inequality~\eqref{eq:p-product-lower}, we have
\[
        \prod_{j = 1}^q(1 - x_{ij})
 \geq
        \prod_{j = 1}^q\max\{1 - p_{ij} - \xi_j,0\}
 \geq
        \prod_{j = 1}^q(1 - p_{ij}) -
        \sum_{j = 1}^q \xi_j
 \geq
        \rho_q -
        \sum_{j = 1}^q \xi_j.
\]
This completes the proof.
\end{proof}

\paragraph{Acknowledgements.}
This work was initiated during the first author's visit to the second author at
Zhejiang Normal University in February 2026. The first author thanks the
second author for his warm hospitality during the visit. C.F. was supported by
the National Natural Science Foundation for Young Scientists of China (Grant
No.~12301435). R.X. was supported by the National Natural Science Foundation
for Young Scientists of China (Grant No.~12401472) and the Zhejiang Provincial
Natural Science Foundation of China (Grant No.~LQN25A010011).

\paragraph{Declaration of generative AI use.} 

During the preparation of this work, the authors used AI tools to assist with probabilistic and asymptotic calculations, exploring possible proof strategies, particularly for Lemma~\ref{lem:robust-array}, and final proofreading. All AI-assisted calculations and suggestions were independently verified and revised by the authors. The authors take full responsibility for the content of the paper.


\begingroup
\small
\begin{thebibliography}{99}

\bibitem{Alon1992}
N.~Alon,
\newblock Choice numbers of graphs: a probabilistic approach,
\newblock \emph{Combin. Probab. Comput.} 1 (1992), 107--114.

\bibitem{Alon2000}
N.~Alon,
\newblock Degrees and choice numbers,
\newblock \emph{Random Structures Algorithms} 16 (2000), 364--368.

\bibitem{BMS2015}
J.~Balogh, R.~Morris and W.~Samotij,
\newblock Independent sets in hypergraphs,
\newblock \emph{J. Amer. Math. Soc.} 28 (2015), 669--709.

\bibitem{BonamyKang2017}
M.~Bonamy and R.~J. Kang,
\newblock List colouring with a bounded palette,
\newblock \emph{J. Graph Theory} 84 (2017), 93--103.

\bibitem{BondyMurty2008}
J.~A. Bondy and U.~S.~R. Murty,
\newblock \emph{Graph Theory},
\newblock Graduate Texts in Mathematics, vol.~244, Springer, London, 2008.

\bibitem{EKT2019}
L.~Esperet, R.~J. Kang and S.~Thomass\'e,
\newblock Separation choosability and dense bipartite induced subgraphs,
\newblock \emph{Combin. Probab. Comput.} 28 (2019), 720--732.

\bibitem{ERT1979}
P.~Erd\H{o}s, A.~L. Rubin and H.~Taylor,
\newblock Choosability in graphs,
\newblock in \emph{Proceedings of the West Coast Conference on Combinatorics, Graph Theory and Computing}, Congressus Numerantium XXVI (1979), 125--157.

\bibitem{FKK2014}
Z.~F\"uredi, A.~Kostochka and M.~Kumbhat,
\newblock Choosability with separation of complete multipartite graphs and hypergraphs,
\newblock \emph{J. Graph Theory} 76 (2014), 129--137.

\bibitem{GuJiangWoodZhu2023}
Y.~Gu, Y.~Jiang, D.~R. Wood and X.~Zhu,
\newblock Refined list version of Hadwiger's conjecture,
\newblock \emph{SIAM J. Discrete Math.} 37 (2023), 1738--1750.

\bibitem{GuZhu2023}
Y.~Gu and X.~Zhu,
\newblock Girth and $\lambda$-choosability of graphs,
\newblock \emph{J. Graph Theory} 103 (2023), 493--501.

\bibitem{Kang2013}
R.~J. Kang,
\newblock Improper choosability and Property B,
\newblock \emph{J. Graph Theory} 73 (2013), 342--353.

\bibitem{KralSgall2005}
D.~Kr\'al' and J.~Sgall,
\newblock Coloring graphs from lists with bounded size of their union,
\newblock \emph{J. Graph Theory} 49 (2005), 177--186.

\bibitem{KTV1998Brooks}
J.~Kratochv{\'\i}l, Zs.~Tuza and M.~Voigt,
\newblock Brooks-type theorems for choosability with separation,
\newblock \emph{J. Graph Theory} 27 (1998), 43--49.

\bibitem{KTV1998Complexity}
J.~Kratochv{\'\i}l, Zs.~Tuza and M.~Voigt,
\newblock Complexity of choosing subsets from color sets,
\newblock \emph{Discrete Math.} 191 (1998), 139--148.

\bibitem{ST2015}
D.~Saxton and A.~Thomason,
\newblock Hypergraph containers,
\newblock \emph{Invent. Math.} 201 (2015), 925--992.

\bibitem{Vizing1976}
V.~G. Vizing,
\newblock Coloring the vertices of a graph in prescribed colors,
\newblock \emph{Diskret. Analiz} 29 (1976), 3--10 (in Russian).

\bibitem{Zhu2020}
X.~Zhu,
\newblock A refinement of choosability of graphs,
\newblock \emph{J. Combin. Theory Ser. B} 141 (2020), 143--164.

\bibitem{ZhuZhu2021}
J.~Zhu and X.~Zhu,
\newblock Chromatic $\lambda$-choosable and $\lambda$-paintable graphs,
\newblock \emph{J. Graph Theory} 98 (2021), 642--652.

\bibitem{ZhuZhu2025}
J.~Zhu and X.~Zhu,
\newblock Minimum non-chromatic-$\lambda$-choosable graphs,
\newblock \emph{J. Graph Theory} 110 (2025), 283--289.

\end{thebibliography}
\endgroup
\end{document}